\documentclass[12pt]{amsart}
\usepackage{amsmath}
\allowdisplaybreaks

\usepackage{array}
\usepackage[a4paper, margin=3cm]{geometry}
\usepackage{multirow}
\usepackage{graphicx}
\usepackage{xcolor}
\usepackage{float}
\usepackage{caption}
\usepackage{amssymb}
\usepackage{tikz}
\usepackage[colorlinks=true,
    linkcolor=teal,      
    citecolor=magenta,       
    urlcolor=magenta,    
    filecolor=green      
]{hyperref}
\usepackage[nameinlink]{cleveref}
\usepackage[utf8]{inputenc}
\usepackage[T1]{fontenc}
\usepackage{lmodern}
\usepackage{indentfirst}
\usepackage{fancyhdr}
\usepackage{stmaryrd}
\usepackage{titlesec}

\makeatletter
\providecommand{\@secnumfont}{}
\providecommand{\@secnumpunct}{. }
\makeatother

\usetikzlibrary{positioning}

\newtheorem{theorem}{\bf Theorem}[section]
\newtheorem{proposition}[theorem]{\bf Proposition}
\newtheorem{lemma}[theorem]{\bf Lemma}
\newtheorem{corollary}[theorem]{\bf Corollary}

\newtheorem*{theorem*}{\bf Theorem}

\newtheorem{remark}[theorem]{\bf Remark}

\newcommand{\field}{K}
\renewcommand{\hom}{\mathrm{Hom}}
\newcommand{\diag}{\mathrm{diag}}
\newcommand{\rank}{\mathrm{rk}}

\title{DEGENERATION ORDER OF $3\times 3$ NILPOTENT MATRIX TUPLES}

\author[M\'aty\'as Domokos]
{M\'aty\'as Domokos}
\address{HUN-REN Alfréd Rényi Institute of Mathematics,
Reáltanoda utca 13-15, 1053 Budapest, Hungary,
ORCID iD: https://orcid.org/0000-0002-0189-8831}
\email{domokos.matyas@renyi.hu}
\author[Botond Mikl\'osi]{Botond Mikl\'osi} 
\address{E\"otv\"os Lor\'and University, 
Pázmány Péter sétány 1/C, 1117 Budapest, Hungary} 
\email{miklosiboti@gmail.com}
\date{}


\begin{document}
\thanks{Partially supported by the Hungarian National Research, Development and Innovation Office,  NKFIH K 138828.}

\begin{abstract}
The degeneration order of simultaneous similarity classes 
of $3\times 3$ nilpotent matrix tuples is determined, and is  
shown to be given by rank conditions. 
\end{abstract}

\maketitle 

\def\lt{\mathrm{B}} 
\def\slt{\mathfrak{n}} 
\def\GL{\mathrm{GL}} 
\def\lhom{\le_{\mathrm{hom}}}
\def\ldeg{\le_{\mathrm{deg}}}
\def\lshom{<_{\mathrm{hom}}}

\section{Introduction}\label{sec:basic definitions}

\subsection{Basic definitions.}

For positive integers $n\ge 2$, $m$,  
denote by $(\field^{n\times n})^m$ the space of $m$-tuples of $n\times n$ matrices 
over an algebraically closed field $\field$.  
The general linear group $\GL_n(\field)$ acts on 
$(\field^{n\times n})^m$ by simultaneous conjugation: 
for $g\in \GL_n(\field)$ and $A=(A_1,\dots,A_m)$ 
we have $g\cdot A=(gA_1g^{-1},\dots,gA_mg^{-1})$. 
We denote by $\mathcal{O}(A)$ the 
$\GL_n(\field)$-orbit of $A$. 
We say that the matrix tuples $A$ and $B$ are \emph{similar} 
if $\mathcal{O}(A)=\mathcal{O}(B)$. 
Write $\overline{\mathcal{O}(A)}$ for the closure of $\mathcal{O}(A)$ in $(\field^{n\times n})^m$ 
with respect to the Zariski topology. 

We shall consider two preorders on $(\field^{n\times n})^m$. 
We write $A\ldeg B$ and say that \emph{$A$ degenerates to $B$} if 
$B\in \overline{\mathcal{O}(A)}$ (or equivalently, 
$\mathcal{O}(B)\subseteq \overline{\mathcal{O}(A)}$). 
Clearly, $\ldeg$ induces a partial ordering on the set of $\GL_n(\field)$-orbits in 
$(\field^{n\times n})^m$, called the \emph{degeneration order}. 

Let us turn to the second preorder. Let 
$R:=\field \langle x_1,\dots,x_m\rangle$ be the free 
associative $\field$-algebra on $m$ generators $x_1,\dots,x_m$ 
(with $1\in R$). 
Given an $f\in R$ and $A\in (\field^{n\times n})^m$ 
it makes sense to substitute in $f$ the non-commuting variables $x_i$ by $A_i$, and we write 
$f(A)$ for the resulting $n\times n$ matrix. 
Moreover, given an arbitrary matrix 
$\varphi=(\varphi_{ij})_{i=1,\dots,k}^{j=1,\dots,l}\in R^{k\times l}$, 
set $\varphi(A):=(\varphi_{ij}(A))_{i=1,\dots,k}^{j=1,\dots,l}$, which is a 
$kn\times ln$ matrix over $\field$, partitioned into $k\times l$ blocks of size $n\times n$ 
(and $\varphi_{ij}(A)$ is the $n\times n$ block in the $(i,j)$ position). 
Denote by $\rank(\varphi(A))$ the \emph{rank} of the $kn\times ln$ matrix $\varphi(A)$ 
over $\field$. 
Now for $A,B\in (\field^{n\times n})^m$ we set 
$A\lhom B$ if $\rank\varphi(A)\ge \rank\varphi(B)$ 
for all possible matrices $\varphi$ over $R$. 
 By theorems of Auslander \cite{Auslander} and Riedtmann \cite{Riedtmann1} 
(see Section~\ref{sec:module theory}), 
$A\lhom B$ and $B\lhom A$ both hold if and only if 
$\mathcal{O}(A)=\mathcal{O}(B)$. So 
$\lhom$ also induces a partial ordering (denoted by $\lhom$ as well) 
on the set of $\GL_n(\field)$-orbits in $(\field^{n\times n})^m$. 
For reasons explained in Section~\ref{sec:module theory}, 
this is called the \emph{hom-order}. 
We shall use the notation $A\lshom B$ if $A\lhom B$ and 
$\mathcal{O}(A)\neq \mathcal{O}(B)$. 

We have the following well known lemma: 

\begin{lemma}\label{lemma:deg implies hom} 
$A\ldeg B$ implies $A\lhom B$ for $A,B\in (\field^{n\times n})^m$.  
\end{lemma}

\begin{proof} 
For any $\varphi\in R^{k\times l}$, $A\in (\field^{n\times n})^m$, and $g\in \GL_n(\field)$, 
$\varphi(g\cdot A)$ belongs to the $\GL_{nk}(\field)\times \GL_{nl}(\field)$-orbit of 
$\varphi(A)$, hence the function $\rank\circ \varphi$ is constant along the $\GL_n(\field)$-orbits in 
$(\field^{n\times n})^m$. On the other hand, for a fixed non-negative integer $r$, the set 
$\{B\in (\field^{n\times n})^m\mid \rank\varphi(B)\le r\}$ is the common zero locus of the determinants 
of the minors of size $(r+1)\times (r+1)$ in $\varphi(B)$, hence is Zariski closed. Thus if 
$B$ belongs to the Zariski closure of the $\GL_n(\field)$-orbit of $A$, then 
$\rank\varphi(B)\le \rank\varphi(A)$. 
\end{proof}

\subsection{Motivation.} 
The study of orbit closures is motivated by several areas of mathematics.  
For example, basic problems of Geometric Complexity Theory ask to decide if an orbit under a linear action of a reductive group (say the general linear group) is contained in the Zariski closure of another orbit, 
see for example \cite{landsberg}. So it is a natural and challenging problem 
to characterize the preorder $\ldeg$. A possible candidate for such a characterization is offered 
by $\lhom$ (see \cite{CH85},\cite{HL03}). However, 
the converse of Lemma~\ref{lemma:deg implies hom} does not hold in general, 
as an example attributed to Carlson in \cite{Riedtmann1} shows (see also \cite{DKMV}, 
and Remark~\ref{remark:n>3} below). 
Note also that the problem of classifying the similarity classes in $(\field^{n\times n})^2$ for all $n$ 
is \emph{wild} in the sense of \cite{Drozd}, but some approaches to it are given in 
\cite{Bongartz-Friedland}, \cite{lebruyn2}, \cite{Belitskii}. 

On the other hand, there are some notable partial converses for Lemma~\ref{lemma:deg implies hom} 
in the theory of finite dimensional modules over a finite dimensional associative algebra. 
Note that for a fixed finite dimensional associative algebra $\Lambda$ generated by $m$ elements, 
the space $(\field^{n\times n})^m$ contains a $\GL_n(\field)$-invariant affine subvariety 
denoted by $\mathrm{Rep}_\Lambda(n)$, such 
that the $\GL_n(\field)$-orbits in $\mathrm{Rep}_\Lambda(n)$ are in bijection with the isomorphism classes of 
$n$-dimensional $\Lambda$-modules. Under certain restrictions for $\Lambda$, it is true that the 
restriction of $\lhom$ to $\mathrm{Rep}_\Lambda(n)$ implies the restriction of $\ldeg$ to 
$\mathrm{Rep}_\Lambda(n)$. This holds if $\Lambda$ is a representation-finite algebra by 
\cite{Zwara:PAMS}, 
or when $\Lambda$ is the path algebra of an extended Dynkin quiver 
by \cite{Bongartz1}. 

An exact module theoretic characterization of when $A\ldeg B$ for 
$A,B\in \mathrm{Rep}_\Lambda(n)$ was found in \cite{Riedtmann1}, \cite{Zwara1}. 
The problem in a more geneal setup is considered in \cite{Forbregd}. 
See \cite{Smalo} and \cite{Nornes} for a survey on the topic. 
 
\subsection{Content of the present paper.} 
The \emph{nullcone} $\mathcal{N}_{n,m}$ of the $\GL_n(\field)$-variety 
$(\field^{n\times n})^m$ consists of the matrix tuples $A$ such that the 
Zariski closure of $\mathcal{O}(A)$ contains the 
zero tuple $\mathbf{0}$ (i.e. $A\ldeg \mathbf{0}$). 
It is also the common zero locus of all homogeneous $\GL_n(\field)$-invariant polynomial functions 
on $(\field^{n\times n})^m$, so it is an affine $\GL_n(\field)$-variety. 
Note that the general notion of the nullcone for a representation of a reductive algebraic group plays 
a distinguished role in geometric invariant theory \cite{MFK94}. 

A central result in our paper is the following partial converse to 
Lemma~\ref{lemma:deg implies hom}: 

\begin{theorem}\label{thm:main} 
For $A,B\in \mathcal{N}_{3,m}$, we have that 
\[A\ldeg B \iff A\lhom B.\]
\end{theorem}

\begin{remark}\label{remark:n>3} 
The analogues of Theorem~\ref{thm:main} for $\mathcal{N}_{n,m}$ with 
$n\ge 4$ does not hold. 
Carlson's example \cite{Riedtmann1} mentioned above provides explicit 
pairs $A,B\in \mathcal{N}_{4,2}$ with $A\lhom B$, such that 
$B\notin \overline{\mathcal{O}(A)}$ (see also \cite{DKMV}). 
So Theorem~\ref{thm:main} shows that this example is in some sense minimal 
possible. 
\end{remark}

The paper is organized as follows. 
 In Section~\ref{sec:module theory} we begin with Lemma~\ref{lemma:hom-rank} stating a module theoretic description 
 of the preorder $\lhom$ due to Riedtmann \cite{Riedtmann1}; we present a new proof fitting to our notation. 
 Although $A\lhom B$ for $A,B\in (\field^{n\times n})^m$ does not necessarily imply $B\in \overline{\mathcal{O}(A)}$ in general, we show the weaker statement that $A\lshom B$ implies $\dim \mathcal{O}(B)<\dim \mathcal{O}(A)$; 
 this is the content of Lemma~\ref{lemma:hom order implies dim ineq}, 
proved by combining Lemma~\ref{lemma:hom-rank} with a result of Bongartz \cite{Bongartz2}. 
Section~\ref{sec:two composition factors} deals with the case of $2\times 2$ matrices, and we prove in 
Corollary~\ref{cor:JHle2} that $\ldeg$ and $\lhom$ coincide on $(\field^{2\times 2})^m$. 
This is deduced from a more general result 
dealing with modules having Jordan-Hölder composition length at most $2$. 
Section~\ref{sec:N32} is devoted to a direct study of $\mathcal{N}_{3,2}$. 
In Proposition~\ref{prop:list of 3x3 nilpotent orbits} we classify 
the similarity classes of $3\times 3$ nilpotent matrix pairs,  
in Proposition~\ref{prop:degenerations} we determine the degeneration order on them 
(see Figure~\ref{figure:hasse}),  
and in Proposition~\ref{prop:hom implies deg in N_{3,2}} we point out that 
the restrictions of $\lhom$ and $\ldeg$ to $\mathcal{N}_{3,2}$ coincide. 
In Section~\ref{sec:N3m} we deduce from Proposition~\ref{prop:hom implies deg in N_{3,2}} 
that $\ldeg$ and $\lhom$ coincide on $\mathcal{N}_{3,m}$ for all $m$, i.e. 
Theorem~\ref{thm:main} holds. 
In Section~\ref{sec:hesselink} we review the Hesselink stratification of $\mathcal{N}_{3,2}$ worked out in 
\cite{lebruyn}, and describe how the $\GL_3(\field)$-orbits listed in Proposition~\ref{prop:list of 3x3 nilpotent orbits} are distributed among the Hesselink strata. In the final Section~\ref{sec:GL3xGL2} 
we take into account the natural action of $\GL_2(\field)$ 
on $(\field^{n\times n})^2$ commuting with the $\GL_n(\field)$-action.  
We show that there are finitely many $\GL_3(\field)\times \GL_2(\field)$-orbits 
in $\mathcal{N}_{3,2}$ and determine the degeneration order for the orbits 
in Proposition~\ref{prop:GL3xGL2}. In fact there is a $2$-dimensional subgroup 
$H$ in $\GL_2(\field)$ such that the number of 
$\GL_3(\field)\times H$-orbits is still finite, their degeneration order is determined 
in Proposition~\ref{prop:GL3xH-orbits}.

\section{Some module theory}\label{sec:module theory} 

For $A\in (\field^{n\times n})^m$ denote by $V_A$ the $R$-module 
$V_A=\field^n$ with module structure $f\cdot v:=f(A)v$ for 
$f\in R$ and $v\in \field^n$. 
Clearly, the $R$-modules $V_A$ and $V_B$ are isomorphic 
if and only if $\mathcal{O}(A)=\mathcal{O}(B)$. 
The relation $\lhom$ is called the hom-preorder because of 
Lemma~\ref{lemma:hom-rank} below. It is due to 
Riedtmann \cite[Section 2.1, Lemma]{Riedtmann1},   
where it is stated and proved using the language of representations of quivers with admissible 
relations. 
Here we give a different proof, suited better to our notation.  

\begin{lemma} \label{lemma:hom-rank}
For $A,B\in (\field^{n\times n})^m$ the following conditions are equivalent: 
\begin{itemize}
    \item[(i)] $A\lhom B$;  
\item[(ii)] $\dim_\field\hom_R(X,V_A) 
\le \dim_\field\hom_R(X,V_B)$  
holds for all 
finite dimensional $R$-modules $X$. 
\item[(iii)] $\dim_\field\hom_R(V_A,X) 
\le \dim_\field\hom_R(V_B,X)$  
holds for all 
finite dimensional $R$-modules $X$.
\end{itemize}
\end{lemma}

\begin{proof} (i)$\Rightarrow$(ii): 
A matrix $\varphi\in R^{k\times l}$ can be interpreted as the $R$-module 
homomorphism $R^k\to R^l$, $[r_1,\dots,r_k]\mapsto [r_1,\dots,r_k]\cdot \varphi$ 
(matrix multiplication on the right hand side). 
Denote by $X$ the cokernel of $\varphi:R^k\to R^l$. 
Applying the functor $\hom_R(-,V_A)$ to the exact sequence 
\begin{equation}\label{eq:exact1} 
R^k\stackrel{\varphi}\to R^l\to X\to 0\end{equation}
we get the exact sequence 
\begin{equation} \label{eq:exact2} 
0\to \hom_R(X,V_A)\to \hom_R(R^l,V_A)\to \hom_R(R^k,V_A).
\end{equation}
Identify $\hom_R(R^k,V_A)$ with $V_A^k$, in the standard way; that is, 
$(v_1,\dots,v_k)\in V_A^k$ corresponds to the $R$-module homomorphism 
$(r_1,\dots,r_k)\mapsto \sum_{i=1}^k r_iv_i$. Similarly, make the identification 
$\hom_R(R^l,V_A)\cong V_A^l$. Then 
$\hom_R(\varphi,V_A)$ gets identified with the $\field$-linear map 
$\field^{nl}=V_A^l\to V_A^k=\field^{nk}$ given by 
multiplication from the left by the matrix $\varphi(A)$. 
It follows by \eqref{eq:exact2} that 
\begin{equation}\label{eq:dim image}
\rank(\varphi(A))=\dim_\field\mathrm{Im}(\hom_R(\varphi,V_A))=
nl-\dim_\field\hom_R(X,V_A).
\end{equation} 
Since any finite dimensional $R$-module $X$ has a finite free resolution as in \eqref{eq:exact1} 
(see \cite[P. 458, Corollary]{lewin}), 
\eqref{eq:dim image} shows that (i) implies (ii).  

(ii)$\Rightarrow$(i): 
Assume that (ii) holds for $A,B\in (\field^{n\times n})^m$.   
Take any matrix $\varphi\in R^{k\times l}$. Consider the exact sequence \eqref{eq:exact1}. 
In particular, $X$ denotes the cokernel  of $\varphi$. 
Formula \eqref{eq:dim image} shows that 
\begin{equation}\label{eq:dim image 2}
\dim_\field\hom_R(X,V_A)-\dim_\field\hom_R(X,V_B)=\rank(\varphi(B))-\rank(\varphi(A)).\end{equation} 
If $X$ is finite dimensional, than we may conclude from \eqref{eq:dim image 2} by assumption (i) that 
$0\ge \rank(\varphi(B))-\rank(\varphi(A))$. 

It remains to deal with the case when $X$ is not finite dimensional. 
The annihilator ideals $\mathrm{ann}_R(V_A)$ and $\mathrm{ann}_R(V_B)$ 
both have finite codimension (at most $n^2$) in $R$. Write $J$ for the intersection of the two annihilators, 
then $J$ also has finite codimension in $R$. Set $\bar X:=X/JX$. Then $\bar X$ is a finitely generated module 
over the finite dimensional algebra $R/J$, hence $\bar X$ is finite dimensional. Moreover, any $R$-module 
homomorphism $X\to V_A$ factors through $\bar X$, and similarly, any homomorphism $X\to V_B$ factors through $\bar X$. 
It follows that $\dim_\field\hom_R(X,V_A)=\dim_\field\hom_R(\bar X,V_A)$ 
and $\dim_\field\hom_R(X,V_B)=\dim_\field\hom_R(\bar X,V_B)$. 
Thus \eqref{eq:dim image 2} yields 
\[\dim_\field\hom_R(\bar X,V_A)-\dim_\field\hom_R(\bar X,V_B)=\rank(\varphi(B))-\rank(\varphi(A)),\] 
implying by (ii) the desired inequality $\rank(\varphi(A))\ge \rank(\varphi(B))$. 

(i)$\Rightarrow$(iii):  
For a $\field$-vector space $X$ write $X^*$ for the dual space 
$\hom_\field(X,\field)$; when $X$ is a left $R$-module, $X^*$ is naturally a right $R$-module. 
Now let $X$ be an arbitrary finite dimensional left $R$-module, 
and take a free resolution 
\begin{equation}\label{eq:exact3} 
R^l\stackrel{\varphi}\to R^k\to X^*\to 0.\end{equation}
This time the free right $R$-modules $R^l$, $R^k$ consist of column vectors with entries from $R$, 
and the map $\varphi$ is multiplication from the left by 
$\varphi\in R^{k\times l}$.  
Apply the functor $-\otimes_RV_A$ to \eqref{eq:exact3} to get the exact sequence 
\begin{equation}\label{eq:exact4}
R^l\otimes_RV_A\to R^k\otimes_RV_A\to X^*\otimes_RV_A\to 0.
\end{equation}
Note that we have the natural isomorphism 
$X^*\otimes_RV_A\cong\hom_R(V_A,X)^*$. 
Moreover, $R^l\otimes_RV_A\cong V_A^l=\field^{nl}$, 
$R^k\otimes_RV_A\cong V_A^k=\field^{nk}$ as $\field$-vector spaces, and 
the map $\varphi\otimes_R \mathrm{id}_{V_A}$ is identified with multiplication by the 
$nk\times nl$ matrix $\varphi(A)$. 
It follows by \eqref{eq:exact4} that 
\begin{align}\label{eq:dim image 3}
    \dim_\field\hom_R(V_A,X)=\dim_\field\hom_R(V_A,X)^*=
\dim_\field(X^*\otimes_RV_A)
\\ \notag =nk-\dim_\field(\mathrm{Im}(\varphi\otimes_R\mathrm{id}_{V_A}))
=nk-\rank(\varphi(A)). \end{align}
Thus $\rank(\varphi(A))\ge \rank(\varphi(B))$ implies 
$\dim_\field\hom_R(V_A,X)\le \dim_\field\hom_R(V_B,X)$. 

(iii)$\Rightarrow$(i): 
Take an arbitrary matrix $\varphi\in R^{k\times l}$. Denote also by $\varphi$ the right $R$-module homomorphism from 
$R^{l\times 1}\to R^{k\times 1}$ given by multiplication from the left by $\varphi$, and let 
$Y$ be the cokernel of $\varphi$, so we have the exact sequence 
\[R^l\to R^k\to Y\to 0.\]
If $Y$ is finite dimensional, then setting $X:=Y^*$, we conclude from 
\eqref{eq:dim image 3} that 
\[\dim_\field\hom_R(V_A,X)-\dim_\field\hom_R(V_B,X)=\rank(\varphi(B))-\rank(\varphi(A)),\] 
and we are done. 

Suppose finally that $Y$ is not finite dimensional. For given $A$ and $B$ denote by $J$ the intersection of the 
of the annihilators ideals of $V_A$ and $V_B$. Then $\bar Y:=Y/JY$ is a finitely generated module over the finite dimensional algebra 
$R/J$, hence is finite dimensional over $\field$,  $Y\otimes_R V_A\cong \bar Y\otimes_R V_A$, 
 $Y\otimes_R V_B\cong \bar Y\otimes_R V_B$. 
Moreover, \eqref{eq:dim image 3} implies that 
\begin{align*}
\rank(\varphi(B))-\rank(\varphi(A))=\dim_\field(\bar Y\otimes_R V_A)-
\dim_\field(\bar Y\otimes_R V_B)
\\ 
=\dim_\field\hom_R(V_A,X)^*-\dim_\field\hom_R(V_B,X)^*,\end{align*}
where $X=\bar Y^*$. 
\end{proof}

Although the converse of Lemma~\ref{lemma:deg implies hom} does not hold in general (i.e. 
$A\lhom B$ does not imply $A\ldeg B$), the following weaker 
statement holds, which is a consequence of Lemma~\ref{lemma:hom-rank} and a result of 
Bongartz \cite{Bongartz2}: 

\begin{lemma}\label{lemma:hom order implies dim ineq}
Let $A,B\in (\field^{n\times n})^m$. 
If $A\lshom B$, then $\dim \mathcal{O}(A)>\dim \mathcal{O}(B)$. 
\end{lemma}

\begin{proof} 
The dimension $\dim\mathcal{O}(A)$ of the locally closed algebraic variety 
$\mathcal{O}(A)$ equals $\dim\GL_n(\field)-\dim\mathrm{Stab}_{\GL_n(\field)}(A)$. 
The stabilizer $\mathrm{Stab}_{\GL_n(\field)}(A)$ is the group of units in the endomorphism algebra 
$\mathrm{End}_R(V_A)$ of the $R$-module $V_A$. 
Thus we have $\dim\mathcal{O}(A)=n^2-\dim_\field\hom_R(V_A,V_A)$ and 
$\dim\mathcal{O}(B)=n^2-\dim_\field\hom_R(V_B,V_B)$. 
Consequently, we have 
\begin{equation}\label{eq:dim hom orb A}
  \dim\mathcal{O}(A)-\dim\mathcal{O}(B)=
  \dim_\field\hom_R(V_B,V_B)-\dim_\field\hom_R(V_A,V_A).
\end{equation}
By Lemma~\ref{lemma:hom-rank}, $A\lhom B$ implies 
\begin{equation}\label{eq:dim hom A}
    \dim_\field\hom_R(V_A,V_A)\le \dim_\field\hom_R(V_A,V_B)\le 
\dim_\field\hom_R(V_B,V_B).\end{equation}
Combining \eqref{eq:dim hom orb A} and \eqref{eq:dim hom A} we conclude that 
$A\lhom B$ implies $\dim\mathcal{O}(A)\ge \dim\mathcal{O}(B)$. 
Note finally that \cite[Lemma 1.2]{Bongartz2} asserts that 
$A\lhom B$ and 
$\dim_\field\hom_R(V_A,V_A)=\dim_\field\hom_R(V_B,V_B)$ 
together imply $V_A\cong V_B$ as $R$-modules (equivalently, 
$\mathcal{O}(A)=\mathcal{O}(B)$). 
This finishes the proof of the implication $A\lshom B\Rightarrow 
\dim \mathcal{O}(A)>\dim \mathcal{O}(B)$.  
\end{proof} 

\section{Modules with at most two composition factors}
\label{sec:two composition factors}

\begin{lemma}\label{lemma:S,T,M} 
    Let $M$ and $N$ be $R$-modules, $\dim_\field M=\dim_\field N$.  
    Assume that $M$ 
    or $N$ has 
    Jordan-Hölder composition length at most $2$, and 
    \begin{equation}\label{eq:hom(X,M)}
      \text{for all submodules }X \text{ of }M\colon \quad   \dim_{\field}\hom_{R}(X, M) \leq \dim_{\field}\hom_{R}(X, N).
    \end{equation}    
    Then $M\cong N$, or $M$ has Jordan-Hölder composition length $2$, and 
    denoting by $S,T$ the composition factors of $M$  ($S\cong T$ is allowed), we have 
    $N\cong S\oplus T$. 
\end{lemma}

\begin{proof}
By Schur's Lemma we have that for a finite dimensional simple $R$-module $X$ and 
for a finite dimensional $R$-module $V$, $\dim_{\field}\hom_{R}(X, V)$ equals 
the multiplicity of $X$ as a direct summand in the socle of $V$. 
Hence applying \eqref{eq:hom(X,M)} for the simple submodules $X$ of $M$, 
we get that $N$ contains a submodule isomorphic to $\mathrm{Soc}(M)$. 
In particular, if $M=\mathrm{Soc}(M)$ is semisimple, then $\mathrm{Soc}(N)$ 
contains a direct summand isomorphic to $M$, implying in turn 
by $\dim_{\field}M=\dim_{\field}N$  that $M\cong N$. 
From now on we assume that $M$ is not semisimple, 
and $M\ncong N$. 

Suppose that $M$ has Jordan-Hölder composition length $2$. 
Then $\mathrm{Soc}(M)\cong S$ and 
$M/\mathrm{Soc}(M)\cong T$ are simple $R$-modules. 
As we explained above, $N$ has a submodule isomorphic to $S$. 
As $\dim_{\field}\hom_R(M,M)>0$, there exists a non-zero 
$\varphi\in \hom_R(M,N)$ by \eqref{eq:hom(X,M)} applied with $X=M$. 
As $\varphi$ is not an isomorphism,  
it has a non-zero kernel, so we have $\ker(\varphi)\supseteq \mathrm{Soc}(M)$, 
and the image of $\varphi$ must be isomorphic to $T$. So $N$ also has a submodule 
isomorphic to $T$. 
If $S\ncong T$, then $N$ has a submodule isomorphic to $S\oplus T$, 
and taking into account 
$\dim_{\field}M=\dim_{\field}S+\dim_{\field}T=\dim_{\field}N$  
we conclude that $N\cong S\oplus T$. 
If $S\cong T$, then . 
$M/\mathrm{Soc}(M)\cong S$.  
Denoting by $\varphi$ 
the natural surjection $M\to M/\mathrm{Soc}(M)$ composed with the 
embedding $\mathrm{Soc}(M)\to M$, we have that $\varphi$ and 
$\mathrm{id}_M$ are linearly independent elements in 
$\hom_R(M,M)$. It follows that 
$2\le \dim_{\field}\hom_R(M,M)\le 
\dim_{\field}\hom_R(M,N)$. 
Since $M\ncong N$, any homomorphism from $M$ to $N$ contains 
$\mathrm{Soc}(M)$ in its kernel, therefore 
$\dim_{\field}\hom_R(S,N)=\dim_{\field}\hom_R(M/\mathrm{Soc}(M),N)=
\dim_{\field}\hom_R(M,N)\ge 2$. 
Thus the multiplicity of $S$ as a direct summand in $N$ is at least $2$, 
implying in turn that $N\cong S\oplus S$. 

Suppose finally that $N$ has Jordan-Hölder composition length at most $2$. 
Since $N$ contains a submodule isomorphic to 
$\mathrm{Soc}(M)$ but 
$N\ncong \mathrm{Soc}(M)$ 
(since otherwise $N\cong M$ by $\dim_\field N=\dim_\field M$), 
we conclude that $\mathrm{Soc}(M)=S$ is simple, and $M$ has a submodule $M'$ 
with $M'\supseteq \mathrm{Soc}(M)$ and $M'/\mathrm{Soc}(M)=T$ simple. 
Then $1\le \dim_\field\hom_R(M',M)\le \dim_\field\hom_R(M',N)$. 
So there exists a non-zero homomorphism $\varphi\in \hom_R(M',N)$. 
It can not be an isomorphism, hence it factors through $T$. 
Moreover, if $S\cong T$, then $\dim_\field\hom_R(M',M)\ge 2$. 
So in the same way as in the above paragraph, we conclude that $N\cong S\oplus T$, which implies 
in turn by $\dim_\field M=\dim_\field N$ 
that $M=M'$. 
\end{proof}

\begin{corollary}
\label{cor:JHle2}
Let  $A,B\in (\field^{n\times n})^m$, and assume that 
the $R$-module $V_A$ or $V_B$ has 
Jordan-Hölder composition length at most $2$. 
Then $A\ldeg B$ if and only if $A\lhom B$.  
\end{corollary}

\begin{proof} By Lemma~\ref{lemma:deg implies hom} it is sufficient to show that 
    that $A \lhom B$ implies $A\ldeg B$. 
    Suppose that $A \lhom B$. 
    Then $\dim_{\field}\hom_{R}(X, V_A) \leq \dim_{\field}\hom_{R}(X, V_B)$ for 
    all finite dimensional $R$-modules $X$ by Lemma~\ref{lemma:hom-rank}. 
    Therefore by Lemma \ref{lemma:S,T,M} we get that $V_B\cong V_A$ (and so 
    $\mathcal{O}(A)=\mathcal{O}(B)$, hence  $A\ldeg B$), 
    or $V_B \cong S \oplus T$, where $S$ and $T$ are the composition factors of $V_A$. Thus $V_B$ is a semi-simple $R$-module. 
    By GIT we know that the closure of $\mathcal{O}(A)$ contains a unique closed orbit 
    (see for example \cite{MFK94}), 
    and by \cite{artin} it corresponds to the direct sum of the composition factors of $V_A$. Thus, $\mathcal{O}(B)$ is the unique closed orbit in the Zariski closure of $\mathcal{O}(A)$. 
    In particular, we have $A\ldeg B$. 
\end{proof}

As a consequence,  for $m$-tuples of $2 \times 2$ matrices, $\ldeg$ and $\lhom$ are equivalent 
(the symbol $\otimes$ below stands for the Kronecker product of matrices):

\begin{corollary}
\label{Cor3.2}
    The following conditions are equivalent for $A,B \in (\field^{2\times 2})^m$:
    \begin{itemize}
        \item[(i)] $A \leq_{\deg} B$;
        \item[(ii)] $A\lhom B$; 
        \item [(iii)] for all $R$-modules $X$ of dimension at most two we have
            \begin{equation*}
            \label{EQ3}
                \dim_{\field}\hom_{R}(X, V_A) \leq \dim_{\field}\hom_{R}(X, V_B);
            \end{equation*}

                \item[(iv)] for all $T = (T_0, ..., T_m) \in (\field^{2m\times 2})^{m+1}$ 
            \begin{equation*}
                \mathrm{rk}(I \otimes T_0 + A_1 \otimes T_1 + ... + A_m \otimes T_m) \geq \mathrm{rk}(I \otimes T_0 + B_1 \otimes T_1 + ... + B_m \otimes T_m).
            \end{equation*}
    \end{itemize}
\end{corollary}

\begin{proof} (i)$\iff$(ii): 
A $2$-dimensional $R$-module has Jordan-Hölder composition length at most $2$, 
hence this is a special case of Corollary~\ref{cor:JHle2}. 

(ii)$\Longrightarrow$(iii): This is a trivial consequence of Lemma~\ref{lemma:hom-rank}. 

(iii)$\Longrightarrow$(i): $\dim_\field V_A= 2$, hence $V_A$ has at most 
two composition factors, and their dimension is at most $2$. 
Our assumption (iii) guarantees that the conditions of Lemma~\ref{lemma:S,T,M} 
hold for $M=V_A$ and $N=V_B$. 
Thus $V_A\cong V_B$ or $V_B\cong S\oplus T$, where  $S,T$ are the composition factors of 
$V_A$, implying $A\ldeg B$ (as it is explained in the proof of Corollary~\ref{cor:JHle2}). 
    
(iii)$\iff$(iv): The computation of $\dim_{\field}\hom_{R}(X, V_A)$ amounts to 
the computation of the dimension of the solution space of a system of homogeneous linear equations 
whose coefficient matrix is of the form 
$I \otimes T_0 + A_1 \otimes T_1 + ... + A_m \otimes T_m$. 
This is explained in \cite[Lemma 3.3]{DKMV} or in \cite{Bongartz-Friedland}. 
\end{proof}

\begin{remark}
It is shown in \cite[Theorem 1.1]{DKMV} that 
for arbitrary $n$ and $A,B\in (\field^{n\times n})^m$, we have 
$\mathcal{O}(A)=\mathcal{O}(B)$ if and only if 
$ \mathrm{rk}(I \otimes T_0 + A_1 \otimes T_1 + ... + A_m \otimes T_m) = \mathrm{rk}(I \otimes T_0 + B_1 \otimes T_1 + ... + B_m \otimes T_m)$ holds for all 
$(T_0,\dots,T_m)\in (\field^{nm\times n})^{m+1}$. 
As explained in \cite{Bongartz-Friedland} (based on \cite{Friedland} and \cite{Bongartz:Comm.Alg}), 
this result plays a key role in their approach to the classification of similarity classes of tuples of matrices, 
which constructs a stra-tification of 
$(\field^{n\times n})^m$ into locally closed subvarieties, and orbit separating invariant morphisms from the strata into affine spaces. A result on the characterization of the equality $\mathcal{O}(A)=\mathcal{O}(B)$ for matrix pairs whose first component has simple spectrum can be found in \cite{Calderon-Lopatin}. 
\end{remark}

\section{\texorpdfstring{$3\times 3$}{3x3} nilpotent matrix pairs}\label{sec:N32}

We call a matrix tuple $A=(A_1,\dots,A_m)\in (\field^{n\times n})^m$ 
\emph{nilpotent} if $A_1,\dots,A_m$ generate a nilpotent subalgebra of $\field^{n\times n}$. 
Denote by $\mathcal{N}_{n,m}$ the set of nilpotent tuples. 
Then $\mathcal{N}_{n,m}$ is the nullcone for the $\GL_n(\field)$-module 
$(\field^{n\times n})^m$ (see e.g. \cite{lebruyn}); that is, $\mathcal{N}_{n,m}$ is the union of those 
$\GL_n(\field)$-orbits whose Zariski closure contains the zero tuple. 

\subsection{Classification.}
The conjugation action of $\GL_3(\field)$ on $\field^{3\times 3}$ induces an 
action on the set of its linear subspaces, and also on the set of its 
associative subalgebras. 

Denote by $\lt_3(\field)$ the subgroup of $\GL_3(\field)$ consisting of 
invertible $3\times 3$ lower triangular matrices, 
and denote by $\slt_3(\field)$ the space of strictly lower triangular matrices.   
Denote by $E_{ij}$ the matrix unit having entry $1$ in the $(i,j)$-position and $0$ everywhere else. 
Easy calculation yields the following: 

\begin{proposition}\label{prop:classification of subspaces}
A complete irredundant list of representatives of the 
\begin{itemize} 
\item[(i)] $\lt_3(\field)$-orbits in 
$\slt_3(\field)$ is 
\[0,\quad E_{21},\quad E_{31},\quad E_{32},\quad E_{21}+E_{32}.\]
\item[(ii)] $\GL_3(\field)$-orbits of the $2$-dimensional subspaces of 
$\slt_3(\field)$ is  
\[\mathrm{Span}\{E_{21},E_{31}\},\ 
\mathrm{Span}\{E_{31},E_{32}\},\ 
\mathrm{Span}\{E_{21},E_{32}\},\ 
\mathrm{Span}\{E_{21}+E_{32},E_{31}\}.\]
\item[(iii)] $\GL_3(\field)$-orbits of the associative subalgebras of 
$\slt_3(\field)$ is 
\[\{0\},\ \mathrm{Span}\{E_{21}\},\ \mathrm{Span}\{E_{31},E_{32}\},\ 
\mathrm{Span}\{E_{21},E_{31}\},\ \mathrm{Span}\{E_{21}+E_{32},E_{31}\},\ 
\slt_3(\field).\]
\end{itemize} 
Moreover, the elementwise stabilizers of the $6$ algebras in (iii) above are pairwise non-conjugate 
in $\GL_3(\field)$. 
\end{proposition}

An element of $\GL_3(\field)$ belongs to the stabilizer 
of $(A_1,\dots,A_m)\in (\field^{3\times 3})^m$ if and only if it 
commutes with each element of the subalgebra of $\field^{3\times 3}$ 
generated by $A_1,\dots,A_m$. 
Therefore by Proposition~\ref{prop:classification of subspaces} (iii), 
there are at most $6$ isotropy types in $\mathcal{N}_{3,m}$ 
(in fact exactly $6$ when $m\ge 2$, since each subalgebra of 
$\slt_3(\field)$ is generated by $2$ elements). 
Let us list the elementwise stabilizers of the $6$ algebras from 
Proposition~\ref{prop:classification of subspaces} (iii): 
\begin{itemize}
    \item $\GL_3(\field) \text{ for }\{0\}$;
    \item $\left\{ \begin{bmatrix}      a & 0 & 0 \\ 
                        b & a & e \\
                        c & 0 & d
        \end{bmatrix} : a, b, c, d, e \in \field \text{ and } a, d \neq 0 \right\}
\text{ for }\mathrm{Span}\{E_{21}\}$;
    \item $\left\{ \begin{bmatrix} a & 0 & 0 \\ 
                        0 & a & 0 \\
                        c & b & a
        \end{bmatrix} : a,b,c\in \field, \ a \neq 0 \right\}\text{ for }\mathrm{Span}\{E_{31},E_{32}\}$;
        \item $\left\{ \begin{bmatrix} a & 0 & 0 \\ 
                        b & a & 0 \\
                        c & 0 & a
        \end{bmatrix} : a,b,c\in \field, \ a \neq 0 \right\}\text{ for }\mathrm{Span}\{E_{21},E_{31}\}$;
        \item $\left\{ \begin{bmatrix} a & 0 & 0 \\ 
                        b & a & 0 \\
                        c & b & a
        \end{bmatrix} : a,b,c\in \field, \ a \neq 0 \right\}\text{ for }\mathrm{Span}\{E_{21}+E_{32},E_{31}\}$;
    \item $\left\{ \begin{bmatrix} a & 0 & 0 \\ 
                        0 & a & 0 \\
                        c & 0 & a
        \end{bmatrix} : a,c\in \field, \ a \neq 0 \right\}\text{ for }\slt_3(\field)$.
\end{itemize}

\begin{proposition}\label{prop:list of 3x3 nilpotent orbits}
The following is a complete irredundant list of representatives of the $\GL_3(\field)$-orbits in 
$\mathcal{N}_{3,2}$: 
\begin{align*}
&A_{\lambda,\mu}:=(E_{21}+E_{32},\lambda(E_{21}+E_{32})+\mu E_{32}), \quad &\lambda,\mu\in \field,\quad \mu\neq 0;\\ 
&A_{\lambda,\infty}:=(E_{21},\lambda E_{21}+E_{32}),\quad &\lambda \in \field ; \\ 
&A_{\infty,\lambda}:=(E_{32},E_{21}+\lambda E_{32}),\quad &\lambda\in \field ; \\ 
&B_{\lambda,\mu}:=(E_{21}+E_{32}, \lambda(E_{21}+E_{32})+\mu E_{31}), \quad &\lambda,\mu\in \field ; \\ 
&B_{\infty,\lambda}:=(\lambda E_{31}, E_{21}+E_{32}),\quad &\lambda\in \field; \\ 
&C:=(E_{21},E_{31}); & \\
&D:=(E_{31},E_{32}); & \\ 
&E_{\lambda}:=(E_{21},\lambda E_{21}), \quad &\lambda\in \field; \\ 
&E_{\infty}:=(0,E_{21}); & \\ 
&O:=(0,0). & 
\end{align*}
\end{proposition}
\begin{proof} 
It follows from the theorem on Jordan normal form that if \linebreak $\dim_\field\mathrm{Span}\{A_1,A_2\}\le 1$ 
for some $A=(A_1,A_2)\in \mathcal{N}_{3,2}$, then 
the $\GL_3(\field)$-orbit of $A$ contains an element of the form 
$B_{\lambda,0}$, $B_{\infty,0}$, $E_\lambda$, $E_\infty$, or $O$. 

Assume next that $A=(A_1,A_2)\in \mathcal{N}_{3,2}$ with 
$\dim_\field\mathrm{Span}\{A_1,A_2\}=2$.  
A well-known application of the Hilbert-Mumford criterion shows that every $\GL_3(\field)$-orbit of $A$ intersects $\slt_3(\field)\oplus \slt_3(\field)$, see for example \cite{kraft}. 
The space $\slt_3(\field)\oplus \slt_3(\field)$ is invariant with respect to the action of the subgroup 
$\lt_3(\field)$. 
By Proposition~\ref{prop:classification of subspaces} (i) the $\lt_3(\field)$-orbit of $A$ 
contains an element $B=(B_1,B_2)$ with  
$B_1\in \{E_{21}, E_{31}, E_{32}, E_{21}+E_{32}\}$. It is easy to check 
that one can transform $B_2$ by an appropriate element in the stabilizer of $B_1$ 
to get a pair of the form 
$A_{\lambda,\mu}$, $A_{\lambda,\infty}$, $A_{\infty,\lambda}$, $B_{\lambda,\mu}$ with $\mu\neq 0$, 
$C$, $D$, or to one of 
$(E_{31},E_{21})$, $(E_{32},E_{31})$, 
$(E_{31},\lambda(E_{21}+E_{32}))$.  Note that the permutation matrices corresponding to the transpositions 
$(23)$ and $(12)$ move $(E_{31},E_{21})$ and $(E_{32},E_{31})$ to $C$ and $D$, whereas 
an appropriate diagonal element of $\GL_3(\field)$ conjugates $(E_{31},\lambda(E_{21}+E_{32}))$ into 
$B_{\infty,\lambda^{-2}}$. This shows that the pairs given in the statement represent all the 
$\GL_3(\field)$-orbits in $\mathcal{N}_{3,2}$. 

It remains to show that no two of the pairs in our statement belong to the same $\GL_3(\field)$-orbit. 
The notation for each pair in our statement consists of a letter and (in most cases) some parameters 
from $\field\cup\{\infty\}$. 
The notation  is chosen so that the pairs $A$ and $B$ are denoted by the same letter 
if the subalgebra of $\slt_3(\field)$ generated by $A_1$ and $A_2$ coincides with the subalgebra of 
$\slt_3(\field)$ generated by $B_1$ and $B_2$ (and hence the stabilizer of $A$ in $\GL_3(\field)$ 
coincides with the stabilizer of $B$ in $\GL_3(\field)$). Moreover, when the letter parts for the 
names of $A$ and $B$ differ, their stabilizers are not even conjugate in $\GL_3(\field)$. 
This implies that two pairs denoted by 
different letters can not belong to the same $\GL_3(\field)$-orbit. 
It remains to show that no two different pairs denoted by the same letter 
belong to the same $\GL_3(\field)$-orbit. 
Table~\ref{Tab2} gives rank conditions to distinguish these orbit pairs.
\end{proof}

\subsection{The degeneration order for \texorpdfstring{$3\times 3$}{3x3} nilpotent matrix pairs.}

\begin{proposition}\label{prop:degenerations}
The Hasse diagram of the degeneration order for $\GL_3(\field)$-orbits 
of nilpotent $3\times 3$ matrix pairs is given in Figure~\ref{figure:hasse} 
(each orbit given in Proposition~\ref{prop:list of 3x3 nilpotent orbits} appears once 
in Figure~\ref{figure:hasse}, and the edges going down from it 
connect it to all its minimal degenerations).   
\end{proposition}
\begin{figure}[ht]
\centering
\caption{Degenerations of $\GL_3(\field)$-orbits}
\vspace{1em}
\begin{tikzpicture}[scale=2, every node/.style={inner sep=2pt}]
  \node (d1) at (-5, 4) {\textcolor{teal}{7}};
  \node (d2) at (-5, 2.5) {\textcolor{teal}{6}};
  \node (d3) at (-5, 1) {\textcolor{teal}{4}};
  \node (d4) at (-5, 0) {\textcolor{teal}{0}};
  
  \draw[dashed] (-4.6, -0.2) -- (-4.6, 4.2);

  \node (A2) at (-3, 4) {$\mathcal{O}(A_{\infty, \lambda})$};
  \node (A) at (-1,4) {$\mathcal{O}(A_{\lambda, \mu})$};
  \node (B) at (-4,2.5) {$\mathcal{O}(B_{\rho, \lambda})$};
  \node (C) at (1,4) {$\mathcal{O}(A_{\lambda, \infty})$};
  \node (F) at (2,2.5) {$\mathcal{O}(B_{\infty, \lambda})$};
  \node (D2) at (-2, 2.5) {$\mathcal{O}(D)$};
  \node (D) at (0,2.5) {$\mathcal{O}(C)$};
  \node (E) at (-2.5,1) {$\mathcal{O}(E_{\rho})$};
  \node (G) at (0.5,1) {$\mathcal{O}(E_{\infty})$};
  \node (O) at (-1,0) {$\mathcal{O}(O)$};

  \draw (A) -- (D);
  \draw (B) -- (E);
  \draw (C) -- (D);
  \draw (F) -- (G);
  \draw (D) -- (E);
  \draw (D) -- (G);
  \draw (E) -- (O);
  \draw (G) -- (O);
  \draw (A2) -- (D2);
  \draw (A2) -- (D);
  \draw (A) -- (D2);
  \draw(C) -- (D2);
  \draw(D2) -- (E);
  \draw(D2) -- (G);
\end{tikzpicture}
\label{figure:hasse}
\end{figure}
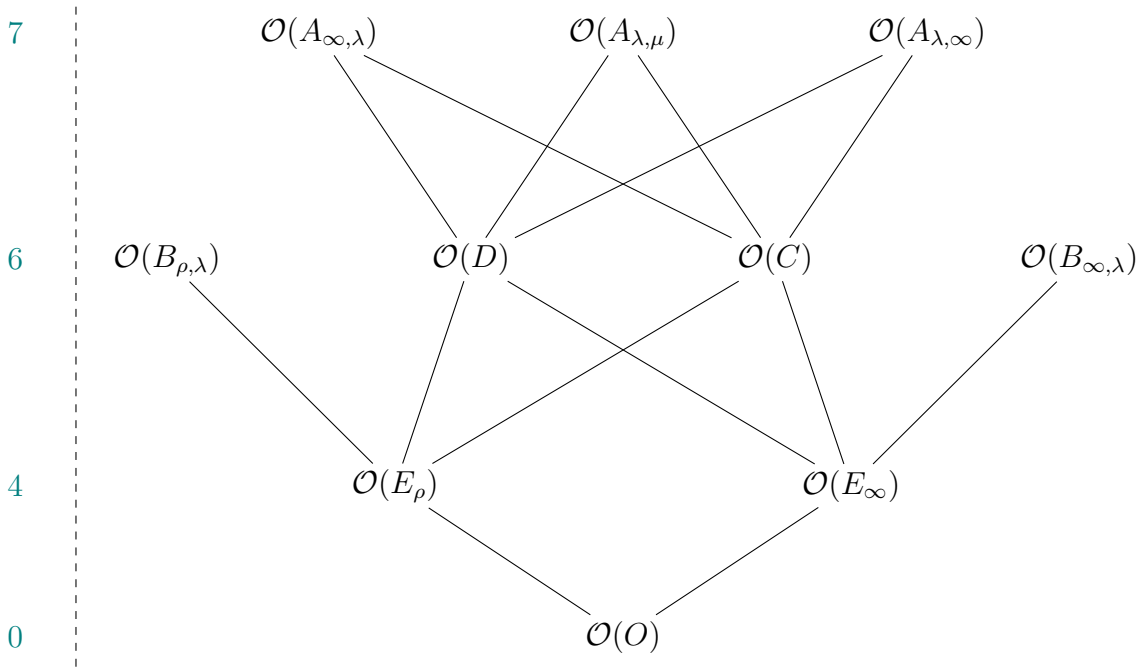

\begin{proof} 
We show that the degenerations indicated by the edges in 
Figure~\ref{figure:hasse} indeed exist in the following way. For each edge going down 
from a pair $A$ to a pair $B$ in Figure~\ref{figure:hasse} we give in Table~\ref{table:denerations} 
a morphism $\field^\times \to \GL_3(\field)$, 
$\varepsilon\mapsto g_\varepsilon$ of affine varieties, such that the induced morphism 
$\field^\times \to (\field^{3\times 3})^2$, $\varepsilon\mapsto g_\varepsilon \cdot A$ 
extends to a morphism $\field\to (\field^{3\times 3})^2$, $\varepsilon\mapsto A_\varepsilon$, 
such that $A_0=B$ (or $A_0$ lies in the 
$\GL_3(\field)$-orbit of $B$). As for each $\varepsilon \neq 0$, 
$A_{\varepsilon}$ lies in the orbit of $A$, we conclude that $B$ belongs to the Zariski closure 
of the orbit of $A$. 
This then implies that the degenerations indicated in Figure~\ref{figure:hasse} all exist. 

Next we need to show that there are no more degenerations. Note that if $\mathcal{O}(B)$ is contained 
in the Zariski closure of 
$\mathcal{O}(A)$, then $\dim\mathcal{O}(B)<\dim\mathcal{O}(A)$ 
(here $\dim$ stands for the dimension in the sense of algebraic geometry). 
Note that $\dim\mathcal{O}(A)$ is the difference of $\dim\GL_3(\field)=9$ and the dimension of 
the stabilizer of $A$. The possible dimensions of the orbits in $\mathcal{N}_{3,2}$ are 
$7,6,4,0$, as indicated in Figure~\ref{figure:hasse}. Inspecting Figure~\ref{figure:hasse} we see that 
it remains to show that the $7$-dimensional orbits do not contain in their Zariski closure the $6$-dimensional 
orbits other than $\mathcal{O}(C)$ or $\mathcal{O}(D)$, and that the $6$-dimensional orbits $\mathcal{O}(B_{\lambda,\mu})$ 
do not contain in their Zariski closure certain $4$-dimensional orbits $\mathcal{O}(E_\nu)$. 
For each such orbit pairs $\mathcal{O}(A)$, $\mathcal{O}(B)$ we give in Table~\ref{table:hom order} an 
element $\varphi\in R=\field\langle x_1,x_2\rangle$ with $\rank\varphi(A)<\rank\varphi(B)$. 
It follows that the relation $A\lhom B$ does not hold, end hence $A\ldeg B$ can not hold by Lemma~\ref{lemma:deg implies hom}. 
\end{proof}

\begin{remark}
Note that the closure of a single orbit may contain infinitely many orbits. 
Indeed, as is shown by Figure~\ref{figure:hasse}, the closure of the orbit of $D$ contains 
the orbits of $E_{\rho}$ for all $\rho\in \field$. 
Note also that $B_{\rho,\lambda}$ degenerates to 
$E_{\nu}$ if and only if $\nu=\rho$. 
\end{remark}

\begin{proposition} \label{prop:hom implies deg in N_{3,2}}
If $A\lhom B$ for some $A,B\in \mathcal{N}_{3,2}$  then $A\ldeg B$. 
\end{proposition}

\begin{proof} 
Assume that $A\lhom B$. If $B\lhom A$ as well, than $\mathcal{O}(A)=\mathcal{O}(B)$ by 
\cite{Auslander} and Lemma~\ref{lemma:hom-rank}, as we pointed out before. 
So it remains to deal with the case that $A\lshom B$. Then 
$\dim\mathcal{O}(A)>\dim\mathcal{O}(B)$ by Lemma~\ref{lemma:hom order implies dim ineq}. 
For such pairs $A,B$ in the proof of Proposition~\ref{prop:degenerations} we demonstrated 
with Table~\ref{table:hom order} that 
if $A$ does not degenerate to $B$, then we do not have $A\lhom B$. 
Thus our statement holds. 
\end{proof} 

\section{\texorpdfstring{$\mathcal{N}_{3,m}$}{N3m} for arbitrary \texorpdfstring{$m$}{m}}
\label{sec:N3m}

For positive integers $d,m,n$ with $d<m$ and 
$\varphi_j\in \field\langle x_1,\dots,x_d\rangle$ 
($j=d+1,\dots,m$) 
set 
\[(\field^{n\times n})^m_\varphi
:=\{A\in (K^{n\times n})^m\mid A_j
=\varphi_j(A_1,\dots,A_d)\text{ for }j=d+1,\dots,m\}.\]
We shall use the following obvious fact:

\begin{lemma} \label{lemma:G-iso}
$(\field^{n\times n})^m_\varphi$ is a Zariski closed 
$\GL_n(\field)$-subvariety of $(\field^{n\times n})^m$, 
and the projection 
$\pi:(\field^{n\times n})^m_\varphi\to (\field^{n\times n})^d$,  
$(A_1,\dots,A_m)\mapsto (A_1,\dots,A_d)$ 
is an isomorphism of $\GL_n(\field)$-varieties. 
In particular, for $A\in (\field^{n\times n})^m_\varphi$ and 
$B\in (\field^{n\times n})^m$ we have 
$A\ldeg B$ if and only if $B\in (\field^{n\times n})^m_\varphi$ and 
$\pi(A)\ldeg \pi(B)$ in $(\field^{n\times n})^d$. 
\end{lemma}

\begin{lemma}\label{lemma:hom implies phi}
Let $A\in (\field^{n\times n})^m_\varphi$ and $B\in (\field^{n\times n})^m$. 
If $A\lhom B$, then $B\in (\field^{n\times n})^m_\varphi$.    
\end{lemma}

\begin{proof}
    By $A\lhom B$ we have 
    \[0=\rank(A_j-\varphi_j(A_1,\dots,A_d))\ge \rank(B_j-\varphi(B_1,\dots,B_d)), \]
    implying that $B_j=\varphi(B_1,\dots,B_d)$ for $j=d+1,\dots,m$, i.e. 
$B\in (\field^{n\times n})^m_\varphi$. 
\end{proof}

\begin{lemma} \label{lemma:2 generators}
For any $A=(A_1,\dots,A_m)\in \slt_3(\field)^m$ there exist 
$1\le i,j\le m$ such that $A_i$ and $A_j$ generate 
the same associative subalgebra of $\slt_3(\field)$ as $A_1,\dots,A_m$. 
\end{lemma} 

\begin{proof}
A maximal linearly independent subset of $\{A_1,\dots,A_m\}$ generates the same associative subalgebra of 
$\slt_3(\field)$ as $A_1,\dots,A_m$. Therefore we are done, unless the components of $A$ span 
$\slt_3(\field)$, that we assume from now on. If $A$ has a rank $2$ component $A_i$, then 
$A_i$ and $A_i^2$ are linearly independent. Moreover, there exists an $A_j$ not contained in 
the subspace spanned by $A_i$ and $A_i^2$. Then necessarily $A_i, A_i^2,A_j$ span 
$\slt_3(\field)$, so $A_i$ and $A_j$ generate $\slt_3(\field)$ as an algebra. 
Assume finally that no component of $A$ has rank $2$. Then $A$ has a component $A_i$ in 
$\mathrm{Span}\{E_{21},E_{31}\}\setminus\mathrm{Span}\{E_{31}\}$, and a component 
$A_j\in \mathrm{Span}\{E_{32},E_{31}\}\setminus\mathrm{Span}\{E_{31}\}$. 
Then $A_i,A_j,(A_i+A_j)^2$ are linearly independent, hence they span $\slt_3(\field)$. 
\end{proof}

\begin{theorem}
    For $A,B\in \mathcal{N}_{3,m}$ we have $A\ldeg B$ if and only if 
    $A\lhom B$. 
\end{theorem}

\begin{proof} 
By Lemma~\ref{lemma:deg implies hom} we know that $A\ldeg B$ implies $A\lhom B$. 
To prove the reverse implication, assume that $A\lhom B$. 
The case $m\le 2$ is covered by Proposition~\ref{prop:hom implies deg in N_{3,2}}. 
It remains to deal with the case $m\ge 3$. By Lemma~\ref{lemma:2 generators} and by symmetry,  
we may assume that $A_1,A_2$ generate the subalgebra of $\slt_3(\field)$ generated by all the components of $A$. 
This means that for $j=3,\dots,m$ there exist $\varphi_j\in \field\langle x_1,x_2\rangle$ with 
$A_j=\varphi_j(A_1,A_2)$. That is, $A\in (\field^{3\times 3})^m_\varphi$. 
Then $A\lhom B$ implies by Lemma~\ref{lemma:hom implies phi} that $B\in (\field^{3\times 3})^m_\varphi$. 
Moreover, it is immediate from the definition of $\lhom$ that $A\lhom B$ implies $\pi(A)\lhom \pi(B)$, where 
$\pi$ is the projection mapping an $m$-tuple of $3\times 3$ matrices to the pair consisting of the first two matrix 
components. By the special case $m=2$ of our statement we know that $\pi(A)\lhom \pi(B)$ implies 
$\pi(A)\ldeg \pi(B)$. As both $A,B$ belong to $(\field^{3\times 3})^m_\varphi$, by Lemma~\ref{lemma:G-iso} we conclude 
that $A\ldeg B$, as desired. 
\end{proof}

\section{The Hesselink stratification of \texorpdfstring{$\mathcal{N}_{3,2}$}{N32}} \label{sec:hesselink}

General results of \cite{kempf}, \cite{ness}, \cite{Hesselink:1}, \cite{Hesselink:2}, 
\cite{kirwan}, \cite{kempf-ness} were applied to the study of the $\GL_n(\field)$-variety 
$\mathcal{N}_{n,m}$ in \cite{lebruyn}. 
Recall that the Hesselink stratification of $\mathcal{N}_{n,m}$ is a partition 
$\mathcal{N}_{n,m}=\bigsqcup_{\beta\in \mathcal{B}}S_\beta$ 
into the disjoint union of finitely many locally closed smooth $\GL_n(\field)$-subvarieties, such that the Zariski closure of any stratum $S_\beta$ is the union of  certain strata. So we have a partial order $\ge$ on $\mathcal{B}$ such that 
the Zariski closure $\overline{S}_\beta=\cup_{\gamma\ge \beta}S_\gamma$; we call this the \emph{degeneration order of the Hesslink strata}.  

The set $\mathcal{B}$ parametrizing the Hesselink strata is a subset of 
a combinatorially defined finite set $\widetilde{\mathcal{B}}$ given in \cite{Hesselink:2}. The elements of $\widetilde{\mathcal{B}}$ 
are increasing sequences of length $n$, consisting of rational numbers. 
Le Bruyn in \cite{lebruyn} associates with each $\beta\in \widetilde{\mathcal{B}}$ 
the following data: a parabolic subgroup $P_\beta$ of 
$\GL_n(\field)$ (consisting of all the matrices that have a zero entry 
in some specified positions above the main diagonal), 
a $P_\beta$-invariant coordinate subspace $U_\beta$ of the space 
$\mathfrak{n}_n(\field)$ of strictly lower triangular $n\times n$ matrices, 
and a coordinate subspace $W_\beta$ of $U_\beta$, which is invariant under 
the action of the Levi subgroup $L_\beta=\prod_{j=1}^d\GL_{n_j}(\field)$ 
of $P_\beta$ (here $\sum_{j=1}^d n_j=n$), and an integer vector 
$\theta_\beta\in\mathbb{Z}^d$. 
The $L_\beta$-module $W_\beta$ turns out to be isomorphic to the space of representations with dimension vector $(n_1,\dots,n_d)$ of some quiver $Q_\beta$.  
It is pointed out in \cite{lebruyn} that $\beta$ belongs to the subset $\mathcal{B}$ of 
$\widetilde{\mathcal{B}}$ if and only if $Q_\beta$ has a $\theta_\beta$-semistable representation of dimension vector $(n_1,\dots,n_d)$ in the sense of \cite{king}. 
Write $W_\beta^{ss}$ for the subset of $\theta_\beta$-semistable points in $W_\beta$. 
It is a non-empty Zariski open subset of $W_\beta$ for each $\beta\in\mathcal{B}$ 
(in fact it is the complement in $W_\beta$ of the common zero locus of all relative $L_\beta$-invariant polynomial functions on $W_\beta$ whose weight is a positive integer multiple of $\theta_\beta$). Recall that $W_\beta$ is a coordinate subspace 
of $U_\beta$. Denote by $\pi:U_\beta\to W_\beta$ the canonical projection, and 
set $U_\beta^{ss}:=\pi^{-1}(W_\beta^{ss})$. 
Now we have 
\[S_\beta=\GL_n(\field)\cdot U_\beta^{ss},\]
and the action map $\GL_n(\field)\times U_\beta^{ss}$ factors through an isomorphism 
\[\GL_n(\field)\times_{P_\beta} U_\beta^{ss}\cong S_\beta.\]
In particular, each $\GL_n(\field)$-orbit in $S_\beta$ intersects 
$U_\beta^{ss}$ in a unique $P_\beta$-orbit.  

Let us turn to the results on $\mathcal{N}_{3,2}$ in \cite{lebruyn}. 
In that case $\widehat{\mathcal{B}}=\mathcal{B}$ has $5$ elements: 

\[\begin{array}{|c|c|c|c|c|}
\hline
 \beta_1& \beta_2&\beta_3&\beta_4&\beta_5 \\ \hline \hline 
-1& -2/3& -1/3& -1/2& 0 \\
\hline
0 & 1/3& -1/3& 0& 0 \\
\hline
1& 1/3& 2/3 & 1/2& 0 \\
\hline
\end{array}\]

\vspace{4pt}

The Hasse diagram of the degeneration order on the Hesselink strata is given on 
Figure~\ref{figure:hasse2}.  

\begin{figure}[ht] 
\centering
\caption{Degeneration order on the Hesselink strata}
\vspace{1em}
\begin{tikzpicture}[scale=2, every node/.style={inner sep=2pt}]

  \node (S1) at (-1, 4) {$S_{\beta_{1}}$};
  \node (S2) at (-2, 3) {$S_{\beta_{2}}$};
  \node (S3) at (0, 3) {$S_{\beta_{3}}$}; 
  \node (S4) at (-1, 2) {$S_{\beta_{4}}$};
  \node (S5) at (-1, 1) {$S_{\beta_{5}}$};

  \draw (S1) -- (S2);
  \draw (S1) -- (S3);
  \draw (S2) -- (S4);
  \draw (S3) -- (S4);
  \draw (S4) -- (S5);
\end{tikzpicture}
\label{figure:hasse2}
\end{figure}
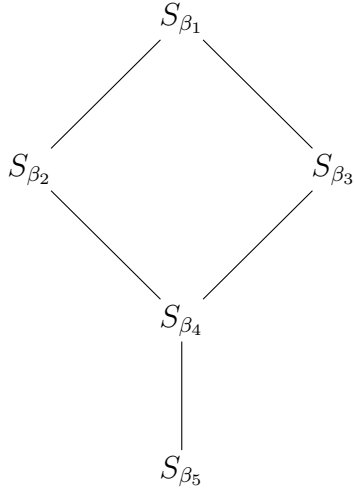

A $1$-parameter subgroup of $\GL_3(\field)$ is a destabilizing 
subgroup for $A=(A_1,A_2)\in\mathcal{N}_{3,2}$ 
if and only if it is a destabilizing subgroup for each matrix in the subalgebra of $\mathfrak{n}_3(\field)$ generated by $A_1$ and $A_2$. 
Therefore if the components of $A=(A_1,A_2)\in\mathcal{N}_{3,2}$ generate the same subalgebra of $\mathfrak{n}_3(\field)$ as the components of 
$B=(B_1,B_2)\in\mathcal{N}_{3,2}$, then $A$ and $B$ belong to the same 
Hesslink stratum, since $A$ and $B$ have the same sets of destabilizing subgroups. 
It follows that the orbit representatives in Proposition \ref{prop:list of 3x3 nilpotent orbits} denoted by the same letter from $\{A,B,C,D,E,O\}$ belong to the same Hesselink stratum. 
The exact distribution in the Hesselink strata of the $\GL_3(\field)$-orbit representatives in $\mathcal{N}_{3,2}$ given in Proposition~\ref{prop:list of 3x3 nilpotent orbits} 
is as follows:  

\begin{proposition} 
We have $S_{\beta_5}=\{O\}$, 
$S_{\beta_4}=\GL_3(\field)\cdot \{E_\lambda,E_\infty\mid \lambda\in\field\}$, 
 $S_{\beta_3}=\GL_3(\field)\cdot \{D\}$,   
$S_{\beta_2}=\GL_3(\field)\cdot \{C\}$, 
and 
$S_{\beta_1}=\GL_3(\field)\cdot \{A_{\lambda,\mu}, 
A_{\lambda,\infty},A_{\infty,\lambda},B_{\lambda,\mu'},
B_{\infty,\lambda}\mid \lambda,\mu,\mu' \in \field, \mu\neq 0\}$. 
\end{proposition}

\begin{proof} 
$S_{\beta_5}$ is the zero orbit $O$. 
Following the approach of \cite{lebruyn} one can easily see that 
the orbit representatives of type 
\begin{itemize}
    \item $A$ belong to $W_{\beta_1}^{ss}$, 
    \item $C$ belong to $W_{\beta_2}^{ss}=U_{\beta_2}^{ss}$, 
    \item $D$ belong to $W_{\beta_3}^{ss}=U_{\beta_3}^{ss}$, 
    \item $E$ belong to $W_{\beta_4}^{ss}=U_{\beta_4}^{ss}$. 
\end{itemize}
Moreover, the points $B_{\lambda,0}$ and $B_{\infty,0}$ also belong to 
$W_{\beta_1}^{ss}$. Finally, the points $B_{\lambda,\mu'}$ with $\mu' \neq 0$ are contained in $U_{\beta_1}^{ss}\setminus W_{\beta_1}^{ss}$.     \end{proof}

\section{Some other group actions}\label{sec:GL3xGL2}

The general linear group $\GL_2(\field)$ acts on $(\field^{3\times 3})^2$ as follows: 
for $A=(A_1,A_2)\in (\field^{3\times 3})^2$ and 
$g=\begin{bmatrix}g_{11}&g_{12}\\g_{21}&g_{22}\end{bmatrix}$ we set 
$g\cdot A=(g_{11}A_1+g_{12}A_2,g_{21}A_1+g_{22}A_2)$. 
This action commutes with the action of $\GL_3(\field)$, and the two actions combines to an action 
of the direct product $\GL_3(\field)\times \GL_2(\field)$ on $(\field^{3\times 3})^2$. 
Moreover, the subvariety $\mathcal{N}_{3,2}$ is preserved by this action.  

Since the $\GL_3(\field)$-orbits in $\mathcal{N}_{3,2}$ are represented by finitely many 
families of dimension at most $2$, 
it is natural to look for a $2$-dimensional subgroup $H$ in 
$\GL_2(\field)$ such that there are finitely many 
$\GL_3(\field)\times H$-orbits in $\mathcal{N}_{3,2}$. 

\begin{proposition} \label{prop:GL3xH-orbits}
Set $G:=\GL_3(\field)\times H$, where 
\[H:=\left\{ \begin{bmatrix}1&0 \\ \lambda& \mu\end{bmatrix} 
\mid \lambda,\mu\in \field, \mu\neq 0 \right\}.\]
There are $12$ $G$-orbits in $\mathcal{N}_{3,2}$, namely: 
\begin{align*}
G\cdot A_{0,1}&=\GL_3(\field)\cdot 
\{A_{\lambda,\mu}\mid \lambda,\mu\in \field, \mu \neq 0\}\\
G\cdot A_{0,\infty}&=\GL_3(\field)\cdot \{A_{\lambda,\infty}
\mid \lambda\in \field\} \\
G\cdot A_{\infty,0}&=\GL_3(\field)\cdot 
\{A_{\infty,\lambda}\mid \lambda\in\field\} \\
G\cdot B_{0,1}&=\GL_3(\field)\cdot 
\{B_{\lambda,\mu}\mid \lambda,\mu\in\field, \mu\neq 0\} \\
G\cdot B_{0,0}&=\GL_3(\field)\cdot 
\{B_{\lambda,0}\mid \lambda\in \field\} \\
G\cdot B_{\infty,0}&=\GL_3(\field)\cdot B_{\infty,0} \\ 
G\cdot B_{\infty,1}&=\GL_3(\field)\cdot 
\{B_{\infty,\lambda}\mid \lambda\in \field, \lambda\neq 0\} \\
G\cdot C&=\GL_3(\field)\cdot C\\ 
G\cdot D&=\GL_3(\field)\cdot D\\
G\cdot E_0&=\GL_3(\field)\cdot \{E_{\lambda}\mid \lambda\in \field\} \\
G\cdot E_\infty&=\GL_3(\field)\cdot E_\infty \\
G\cdot O&= O
\end{align*}
Moreover, the Hasse diagram of the degeneration order for the 
$G$-orbits is given in Figure~\ref{figure:GL3xH}. 
\end{proposition}

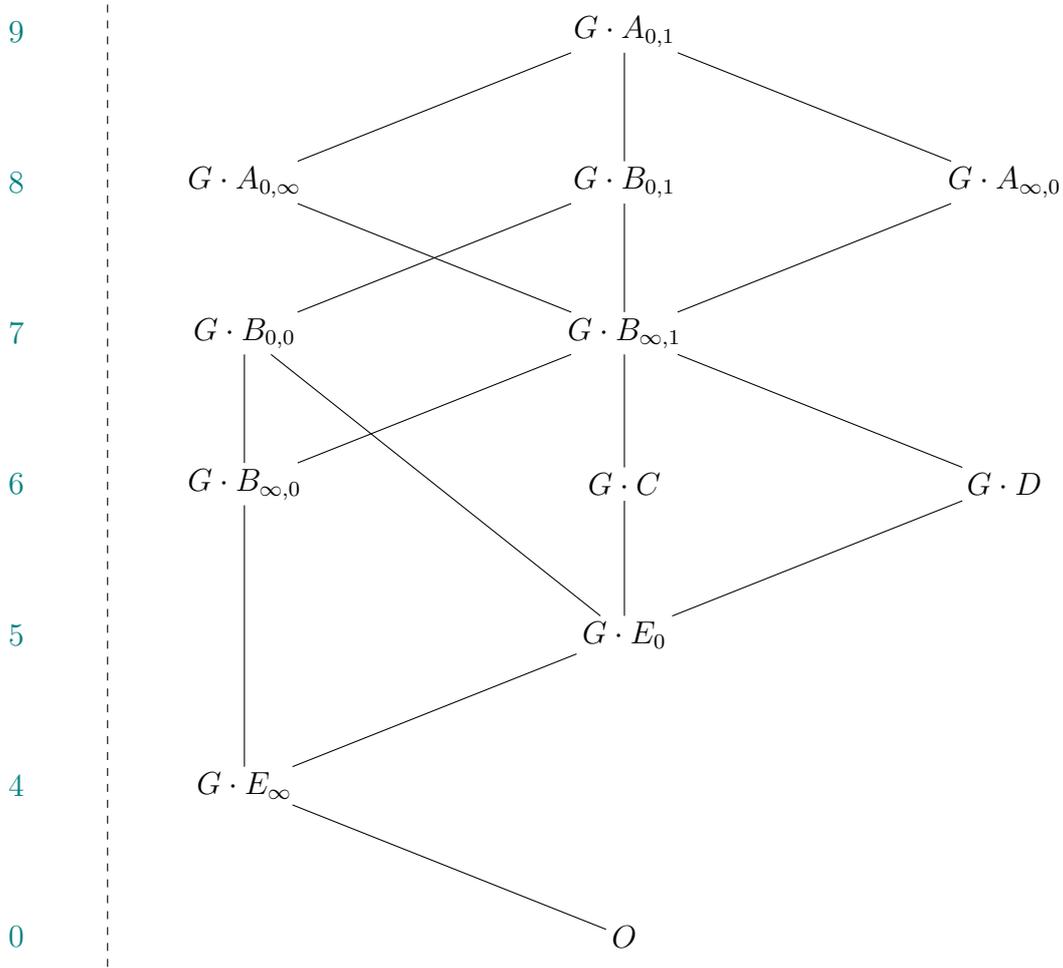
\begin{figure}[ht]
\centering
\caption{Degenerations of $\GL_3(\field)\times H$-orbits}
\vspace{1em}
\begin{tikzpicture}[scale=2, every node/.style={inner sep=2pt}] 
  \node (d1) at (-5, 3) {\textcolor{teal}{9}};
  \node (d2) at (-5, 2) {\textcolor{teal}{8}};
  \node (d3) at (-5, 1) {\textcolor{teal}{7}};
  \node (d4) at (-5, 0) {\textcolor{teal}{6}};
  \node (d5) at (-5,-1) {\textcolor{teal}{5}};
  \node (d6) at (-5,-2) {\textcolor{teal}{4}};
  \node (d7) at (-5, -3) {\textcolor{teal}{0}};

  \draw[dashed] (-4.4, -3.2) -- (-4.4, 3.2);

  \node (A01) at (-1,3) {$G\cdot A_{0,1}$};
  \node (A0inf) at (-3.5, 2) {$G\cdot A_{0, \infty}$};
  \node (Ainf0) at (1.5, 2) {$G\cdot A_{\infty, 0}$};
  \node (B01) at (-1, 2) {$G\cdot B_{0, 1}$};
  \node (Binf1) at (-1, 1) {$G\cdot B_{\infty,1}$};
  \node (B00) at (-3.5, 1) {$G\cdot B_{0,0}$};
  \node (C) at (-1,0) {$G\cdot C$};
  \node (D) at (1.5, -0) {$G\cdot D$};
  \node (Binf0) at (-3.5, -0) {$G\cdot B_{\infty, 0}$};
  \node (E0) at (-1,-1) {$G\cdot E_0$};
  \node (Einf) at (-3.5,-2) {$G\cdot E_\infty$};
  \node (O) at (-1,-3) {$O$};

  \draw(A01) -- (A0inf);
  \draw(A01) -- (Ainf0);
  \draw(A01) -- (B01);
  \draw(B01) -- (B00);
  \draw(B01) -- (Binf1);
  \draw(A0inf) -- (Binf1);
  \draw(Ainf0) -- (Binf1);
  \draw(Binf1) -- (C);
  \draw(Binf1) -- (D);
  \draw(Binf1) -- (Binf0);
  \draw(B00) -- (Binf0);
  \draw(B00) -- (E0);
  \draw(Binf0) -- (Einf);
  \draw(C) -- (E0);
  \draw(D) -- (E0);
  \draw(E0) -- (Einf);
  \draw (Einf) -- (O);
  
\end{tikzpicture}
\label{figure:GL3xH}
\end{figure}

\begin{proof} 
This can be easily deduced from Proposition~\ref{prop:list of 3x3 nilpotent orbits}. 
Indeed, we have $H\cdot A_{0,1}=\{A_{\lambda,\mu}\mid \lambda,\mu\in \field,\mu\neq 0\}$, 
hence $G\cdot A_{0,1}=\GL_3(\field)\cdot 
\{A_{\lambda,\mu}\mid \lambda,\mu\in \field, \mu \neq 0\}$. 
Clearly $\{A_{\lambda,\infty}\mid \lambda\in\field\}
\subseteq H\cdot A_{0,\infty}$, and it is easy to deduce from 
Proposition~\ref{prop:list of 3x3 nilpotent orbits} that 
$H\cdot A_{0,\infty}\subseteq \GL_3(\field)\cdot 
\{A_{\lambda,\infty}\mid \lambda\in\field\}$. 
This shows that 
$G\cdot A_{0,\infty}=\GL_3(\field)\cdot 
\{A_{\lambda,\infty}\mid \lambda\in \field\}$. 
Similar considerations identify which $\GL_3(\field)$-orbits belong to 
the $G$-orbit of each of $A_{\infty,0}$, 
$B_{0,1}$, $B_{0,0}$, $B_{\infty,1}$, $B_{\infty,0}$, 
$C$, $D$, $E_0$, $E_\infty$, $O$. 
Proposition~\ref{prop:list of 3x3 nilpotent orbits} implies then that 
all the $G$-orbits were exhausted. 

The variety $\mathcal{N}_{3,2}$ is irreducible, so it is the closure of 
the maximal $9$-dimensional $G$-orbit. 
In particular, the closure of the $G$-orbit of $A_{0,1}$ contains $A_{\infty,0}$, $A_{0,\infty}$, $B_{0,1}$.   
We know from Figure~\ref{figure:hasse} that the closure of the $\GL_3(\field)$-orbit of $C$ and $D$ 
contains $E_0$, 
the closure of the $\GL_3(\field)$-orbit of $B_{\infty,0}$ contains $E_\infty$, 
the closure of the $\GL_3(\field)$-orbit of $B_{0,0}$ contains $E_0$, 
and the closure of the $\GL_3(\field)$-orbit of $E_{\infty}$ contains $O$. 
We show that the degenerations indicated by the remaining edges in 
Figure~\ref{figure:GL3xH} indeed exist in the following way. For such an edge going down 
from a pair $A$ to a pair $B$ in Figure~\ref{figure:GL3xH} we give in 
Table~\ref{table:GL3xH-degenerations} 
a morphism $\field^\times \to G$, 
$\varepsilon\mapsto g_\varepsilon$ of affine varieties, such that the induced morphism 
$\field^\times \to (\field^{3\times 3})^2$, $\varepsilon\mapsto g_\varepsilon \cdot A$ 
extends to a morphism $\field\to (\field^{3\times 3})^2$, $\varepsilon\mapsto A_\varepsilon$, 
such that $A_0=B$. As for each $\varepsilon \neq 0$, 
$A_{\varepsilon}$ lies in the orbit of $A$, we conclude that $B$ belongs to the Zariski closure 
of the orbit of $A$. 
This then implies that the degenerations indicated in Figure~\ref{figure:GL3xH} all exist. 

For example, that $A_{0,\infty}$ degenerates to $B_{\infty,1}$ can be seen from 
the following equality 
\[\left(\begin{bmatrix}1&0&0\\ \varepsilon^{-1}&\varepsilon&0\\0&1&\varepsilon\end{bmatrix}, 
\begin{bmatrix}1&0\\ \varepsilon^{-1}&1\end{bmatrix}\right)
\cdot A_{0,\infty}=B_{\infty,1}+\varepsilon(E_{21},0)\text{ for }\varepsilon\in \field\setminus\{0\}. \]
Similarly, the equality 
\[\left(\begin{bmatrix}1&0&0\\0& \varepsilon&0 \\0&1&\varepsilon^2\end{bmatrix},
\begin{bmatrix} 1&0\\ \varepsilon^{-1}&-\varepsilon^{-3}\end{bmatrix}\right)
\cdot B_{0,1}=B_{\infty,1}+\varepsilon(E_{21}+E_{32},0) \text{ for }\varepsilon\in\field\setminus\{0\}\]
shows that $B_{0,1}$ degenerates to $B_{\infty,1}$. 

Finally, to show that there are no more degenerations between the $G$-orbits, 
consider the following Zariski closed $G$-invariant subsets of 
$(\field^{3\times 3})^2$: 
\begin{align*}X&:=\{(A_1,A_2)\in (\field^{3\times 3})^2\mid \rank(A_1)\le 1\} \\
Y&:=\{(A_1,A_2)\in (\field^{3\times 3})^2\mid A_1=0\} \\
Z&:= \{(A_1,A_2)\in (\field^{3\times 3})^2\mid \dim_\field \mathrm{Span}\{A_1,A_2)\}\le 1\}
\end{align*} 
Then $A_{0,\infty}, A_{\infty,0}\in X$ and $B_{0,0}\notin X$; 
$B_{0,0}\in Z$ and $C,D\notin Z$; 
$B_{\infty,0}\in Y$ and $E_0\notin Y$. 
\end{proof}
 
To simplify notation, we set 
$\GL_{3,2}:=\GL_3(\field)\times\GL_2(\field)$. 
The number of $\GL_{3,2}$-orbits in $\mathcal{N}_{3,2}$  is $7$: 

\begin{proposition}\label{prop:GL3xGL2}
The $\GL_{3,2}$-orbits in $\mathcal{N}_{3,2}$ are the following: 
\begin{align*}
    \GL_{3,2}\cdot A_{0,1}&=
    \GL_3(\field)\cdot\{A_{\lambda,\mu},A_{\lambda,\infty},A_{\infty,\lambda}\mid
    \lambda,\mu\in\field,\mu\neq 0\}\\
    \GL_{3,2}\cdot B_{0,1}&=
    \GL_3(\field)\cdot\{B_{\lambda,\mu},B_{\infty,\mu}\mid 
    \lambda,\mu\in\field,\mu\neq 0\}\\
    \GL_{3,2}\cdot B_{0,0}&=
     \GL_3(\field)\cdot\{B_{\lambda,0},B_{\infty,0}\mid \lambda\in\field\}\\
     \GL_{3,2}\cdot C&=\GL_3(\field)\cdot C \\
    \GL_{3,2}\cdot D&=\GL_3(\field)\cdot D \\
    \GL_{3,2}\cdot E_0&=\GL_3(\field)\cdot
    \{E_\lambda,E_\infty\mid \lambda\in \field\}\\
    \GL_{3,2}\cdot O&= O
\end{align*}
Moreover, the Hasse diagram of the degeneration order for the 
$\GL_{3,2}$-orbits is given in Figure~\ref{figure:GL3xGL2}. 
    \end{proposition}

\begin{proof} 
This follows easily from Proposition~\ref{prop:GL3xH-orbits}. 
Indeed, $A_{0,1}$, $A_{0,\infty}$, $A_{\infty,0}$ belong to the same $\GL_2$-orbit, 
$B_{0,1}$ and $B_{\infty,1}$ belong to the same $\GL_2$-orbit, 
$B_{0,0}$ and $B_{\infty,0}$ belong to the same $\GL_2$-orbit, and 
$E_0$ and $E_\infty$ belong to the same $\GL_2$-orbit. 
It follows by Proposition~\ref{prop:GL3xH-orbits}  
that $A_{1,0}$, $B_{1,0}$, $B_{0,1}$, $C$, $D$, $E_0$, $O$ represent all the 
$\GL_{3,2}$-orbits. 
It remains to show that these $7$ elements of $\mathcal{N}_{3,2}$ have distinct 
$\GL_{3,2}$-orbits. 
Two points in the same $\GL_2(\field)$-orbit have the same stabilizer in $\GL_3(\field)$, 
hence the stabilizers of two points in the same $\GL_{3,2}$-orbit are conjugate in 
$\GL_3(\field)$. It follows that 
any two of the above $7$ elements with different letter type belong to different $\GL_{3,2}$-orbits. 
Finally, all elements in the $\GL_{3,2}$-orbit of $B_{0,0}$ consist of pairs 
$(A_1,A_2)$ with $\dim_\field\mathrm{Span}_\field\{A_1,A_2)=1$, 
whereas all elements in the $\GL_{3,2}$-orbit of $B_{0,1}$ consist of pairs 
$(A_1,A_2)$ with $\dim_\field\mathrm{Span}_\field\{A_1,A_2)=2$. So these two orbits are also different. 

Note that 
$Z:= \{(A_1,A_2)\in (\field^{3\times 3})^2\mid \dim_\field \mathrm{Span}\{A_1,A_2)\}\le 1\}$ 
is a Zariski closed $\GL_{3,2}$-invariant subset containing $B_{0,0}$ and not containing $C$ or $D$, hence 
$C$ and $D$ are not contained in the Zariski closure of the $\GL_{3,2}$-orbit of $B_{0,0}$. 
The degeneration order for the $\GL_{3,2}$-orbits follows from this observation, 
Figure~\ref{figure:GL3xH} and the partition of the 
$\GL_{3,2}$-orbits into orbits over of the subgroup $G$ of $\GL_{3,2}$. 
\end{proof}

\begin{remark} 
It is possible to split the $12$ $G$-orbits in Proposition~\ref{prop:GL3xH-orbits} 
into $18$ families of $\GL_3(\field)$-orbits parametrized by $0$, $1$ and $2$-dimensional 
tori. The Hasse diagram of a similar stratification into $18$ strata of $\mathcal{N}_{3,2}$ is given by Kraft in 
\cite[Page 202]{kraft} (without providing the details).  
\end{remark}

\begin{figure}[ht]
\centering
\caption{Degenerations of $\GL_3(\field)\times \GL_2(\field)$-orbits}
\vspace{1em}
\begin{tikzpicture}[scale=2, every node/.style={inner sep=2pt}]  
  \node (d1) at (-5, 3) {\textcolor{teal}{9}};
  \node (d2) at (-5, 2) {\textcolor{teal}{8}};
  \node (d3) at (-5, 1) {\textcolor{teal}{7}};
  \node (d4) at (-5, 0) {\textcolor{teal}{6}};
  \node (d5) at (-5,-1.5) {\textcolor{teal}{5}};
  \node (d6) at (-5,-2.5) {\textcolor{teal}{0}};
 
  \draw[dashed] (-4.4, -2.7) -- (-4.4, 3.2);

  \node (A01) at (-1,3) {$\GL_{3,2}\cdot A_{0,1}$};
  \node (B01) at (-1, 2) {$\GL_{3,2}\cdot B_{0, 1}$};
  \node (B00) at (-3.5, 1) {$\GL_{3,2}\cdot B_{0,0}$};
  \node (C) at (-1,0) {$\GL_{3,2}\cdot C$};
  \node (D) at (1.5, 0) {$\GL_{3,2}\cdot D$};  
  \node (E) at (-1,-1.5) {$\GL_{3,2}\cdot E_0$};  
  \node (O) at (-1,-2.5) {$O$};

  \draw(A01) -- (B01);
  \draw(B01) -- (B00);
  \draw(B01) -- (C);
  \draw(B01) -- (D);
  \draw(B00) -- (E);
  \draw(C) -- (E);
  \draw(D) -- (E);
  \draw(E) -- (O);
  
\end{tikzpicture}
\label{figure:GL3xGL2}
\end{figure}
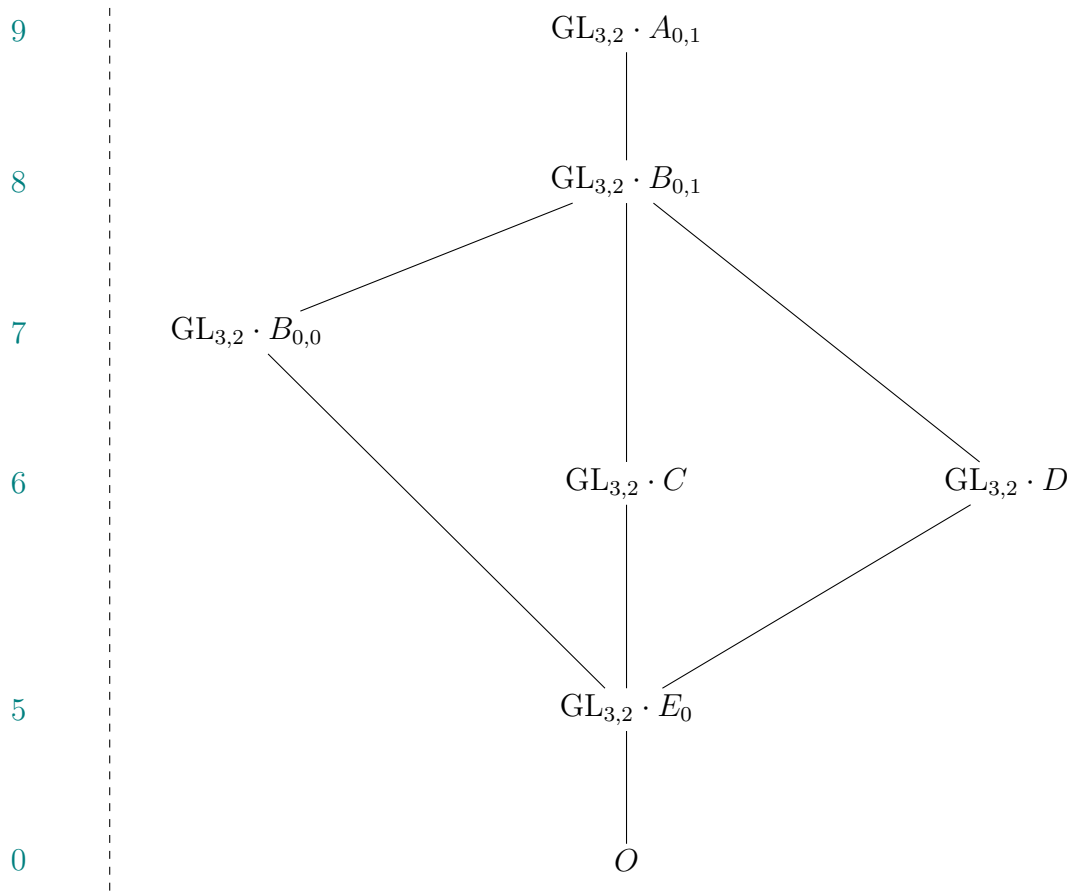

\newpage

\section{Tables}

\begin{table*}[ht]
\centering
\small
\caption{Orbit separation with rank inequalities}
\label{Tab2}
{
\renewcommand{\arraystretch}{2} 
\setlength{\tabcolsep}{10pt} 
\begin{tabular}{|c||c||c||c||c||c|}
\hline
\textbf{\#} & $A$ & $B$ & $\varphi \in \field \langle x_1, x_2\rangle$ & $\rank \varphi(A)$ & $\rank \varphi(B)$ \\
\hline
\hline
\multirow{2}{*}{A1} & \multirow{2}{*}{$A_{\lambda, \mu}$} & \multirow{2}{*}{$A_{\lambda^{\prime}, \mu^{\prime}}$} & $x_2 - \lambda x_1 \quad \text{ if } \lambda \neq \lambda^{\prime}$ & $1$ & $2$ \\
\cline{4-6}
& & & $x_2 - \lambda x_1 - \mu x_1 \quad \text{ if } \lambda = \lambda^{\prime}$ & $1$ & $2$ \\
\hline
A2 & $A_{\lambda, \infty}$ & $A_{\lambda^{\prime}, \infty}$ & $x_2 - \lambda x_1$ & $1$ & $2$ \\
\hline
A3 & $A_{\infty, \lambda}$ & $A_{\infty, \lambda^{\prime}}$ & $x_2 - \lambda x_1$ & $1$ & $2$ \\
\hline
A4 & $A_{\lambda, \mu}$ & $A_{\lambda^{\prime}, \infty}$ & $x_1$ & $2$ & $1$ \\
\hline
A5 & $A_{\lambda, \mu}$ & $A_{\infty, \lambda^{\prime}}$ & 
$x_1$ & $2$ & $1$ \\
\hline
A6 & $A_{\lambda, \infty}$ & $A_{\infty, \lambda^{\prime}}$ & $x_1 x_2$ & $0$ & $1$ \\
\hline
\hline
B1 & $B_{\lambda,\mu}$ & $B_{\lambda',\mu'}$ &  $x_2 - \lambda x_1-\mu x_1^2$ & $0$ & $\ge 1$ \\ 
\hline
B2 & $B_{\infty, \lambda}$ & $B_{\infty, \lambda^{\prime}}$ & $x_1 - \lambda x_2^{2}$ & $0$ & $1$ \\
\hline
B3 & $B_{\lambda, \mu}$ & $B_{\infty, \lambda'}$ & $x_1$ & $2$ & $\le 1$ \\
\hline
\hline
E1 & $E_{\lambda}$ & $E_{\lambda^{\prime}}$ & $x_2 - \lambda x_1$ & $0$ & $1$ \\
\hline
E2 & $E_{\lambda}$ & $E_{\infty}$ & $x_1$ & $1$ & $0$ \\
\hline
\end{tabular}
}
\end{table*}

\newpage

\begin{table*}[ht]
\centering
\small
\caption{Curves for the degenerations}
\label{table:denerations}
\centering
{\renewcommand{\arraystretch}{2}
\setlength{\tabcolsep}{8pt}
\begin{tabular}{|c||c||c||c|}
\hline
\textbf{\#} & $A$ & $B$ & $g_\varepsilon$ \\
\hline\hline
1 & $A_{\lambda, \mu}$ & $D$ & $\diag(1, \varepsilon(\lambda + \mu), \varepsilon) + \varepsilon E_{12} + \lambda E_{21} + E_{32}$ \\
\hline\hline
2 & $A_{\lambda, \infty}$ & $D$ & $\diag(1, \varepsilon, \varepsilon) + \lambda  E_{21} +  E_{32}$ \\
\hline\hline
3 & $A_{\infty, \lambda}$ & $D$ & $\diag(0, \varepsilon \lambda, \varepsilon) + \varepsilon E_{12} + E_{21} + E_{32}$ \\
\hline\hline
4 & $A_{\lambda, \mu}$ & $C$ & $\diag(1, 1+\lambda\mu^{-1}, -\varepsilon)
+ \varepsilon^{-1} \mu^{-1} E_{21} - \mu^{-1} E_{32} +\varepsilon \lambda E_{23}$ \\
\hline\hline
5 & $A_{\lambda, \infty}$ & $C$ & $\diag(\varepsilon, \varepsilon, \varepsilon^2) - E_{21} - \varepsilon^2 \lambda E_{23}$ \\
\hline\hline
6 & $A_{\infty, \lambda}$ & $C$ & $\diag(0, -\varepsilon \lambda, 0) - \varepsilon^{2}(E_{12} + E_{23}) + E_{31} + \varepsilon E_{32}$ \\
\hline\hline
7 & $B_{\lambda, \mu}$ & $E_\lambda$ & $\diag(1, 1, \varepsilon)$ \\
\hline\hline
8 & $D$ & $E_{\lambda}$ & $\diag(\varepsilon, 0, 0) + \varepsilon E_{23} - \lambda E_{31} + E_{32}$ \\
\hline\hline
9 & $C$ & $E_{\lambda}$ & $\diag(1, 1, \varepsilon) +  \lambda E_{23}$ \\
\hline\hline
10 & $D$ & $E_{\infty}$ & $E_{12} + E_{23} + \varepsilon^{-1} E_{31}$ \\
\hline\hline
11 & $C$ & $E_{\infty}$ & $\diag(\varepsilon^{-1}, 1, \varepsilon) + \varepsilon^{-1} E_{23}$ \\
\hline\hline
12 & $B_{\infty, \lambda}$ & $E_{\infty}$ & $\diag(1, 1, \varepsilon)$ \\
\hline
\end{tabular}
}
\end{table*}

\newpage

\begin{table*}[ht]
\centering
\small
\caption{Rank conditions for the hom order}
\label{table:hom order}
\centering
{\renewcommand{\arraystretch}{2}
\setlength{\tabcolsep}{8pt}
\begin{tabular}{|c||c||c||c||c||c|}
\hline
\textbf{\#} & $A$ & $B$ & $\varphi\in \field\langle x_1, x_2\rangle$ & $\rank\varphi(A)$ & $\rank\varphi(B)$ \\
\hline\hline
1 & $A_{\infty, \lambda}$ & $B_{\lambda', \mu}$ & $x_1$ & $1$ & $2$\\
\hline\hline
2 & $A_{\infty, \lambda}$ & $B_{\infty, \lambda'}$ & $x_2 - \lambda x_1$ & $1$ & $2$ \\
\hline\hline
\multirow{2}{*}{3} &
\multirow{2}{*}{$A_{\lambda, \mu}$} & \multirow{2}{*}{$B_{\lambda', \mu'}$} & 
$ x_2 - \lambda x_1  \quad \text{if } \lambda \neq \lambda'$ & $1$ & $2$ \\
\cline{4-6}
& & & $(1 +\frac{\lambda}{\mu}) x_1 - \frac{1}{\mu} x_2 \quad \text{if } \lambda = \lambda'$ & $1$ & $2$ \\
\hline\hline
4 & $A_{\lambda, \mu}$ & $B_{\infty, \lambda'}$ & $x_2 - \lambda x_1$ & $1$ & $2$ \\
\hline\hline
5 & $A_{\lambda, \infty}$ & $B_{\lambda', \mu}$ & $x_1$ & $1$ & $2$ \\
\hline\hline
6 & $A_{\lambda, \infty}$ & $B_{\infty, \lambda'}$ & $x_2 - \lambda x_1$ & $1$ & $2$ \\
\hline\hline
7 & $B_{\lambda, \mu}$ & $E_{\lambda^{\prime}}$ & $x_2 - \lambda x_1 - \mu x_1^{2}$ & $0$ & $1$\\
\hline\hline
8 & $B_{\lambda, \mu}$ & $E_{\infty}$ & $x_2 - \lambda x_1 - \mu x_1^{2}$ & $0$ & $1$\\
\hline\hline
\multirow{2}{*}{9} &
\multirow{2}{*}{$B_{\infty, \lambda}$} & \multirow{2}{*}{$E_{\lambda^{\prime}}$} & 
$ x_1  \quad \text{if } \lambda = 0$ & $0$ & $1$ \\
\cline{4-6}
& & & $x_1 - \lambda x_2^{2} \quad \text{if } \lambda \neq 0$ & $0$ & $1$ \\
\hline
\end{tabular}
}
\end{table*}

\newpage

\begin{table*}[ht]
\centering
\small
\caption{Curves for the $\GL_3(\field)\times H$-degenerations}
\label{table:GL3xH-degenerations}
\centering
{\renewcommand{\arraystretch}{1.3}
\setlength{\tabcolsep}{8pt}
\begin{tabular}{|c||c||c||c|}
\hline
\textbf{\#} & $A$ & $B$ & $g_\varepsilon$ \\
\hline\hline
1 & $A_{0, \infty}$ & $B_{\infty,1}$ & $\left(\begin{bmatrix}1&0&0\\ \varepsilon^{-1}&\varepsilon&0\\0&1&\varepsilon\end{bmatrix}, 
\begin{bmatrix}1&0\\ \varepsilon^{-1}&1\end{bmatrix}\right)$ 
\\
\hline
2 & $B_{0, 1}$ & $B_{0,0}$ & $\left(\diag(1, 1, 1), 
\begin{bmatrix}1&0\\ 0&\varepsilon\end{bmatrix}\right)$
\\
\hline
3 & $B_{0,1}$ & $B_{\infty,1}$ & 
$\left(\begin{bmatrix}1&0&0\\0& \varepsilon&0 \\0&1&\varepsilon^2\end{bmatrix},
\begin{bmatrix} 1&0\\ \varepsilon^{-1}&-\varepsilon^{-3}\end{bmatrix}\right)$\\ 
\hline
4 & $A_{\infty, 0}$ & $B_{\infty,1}$ & 
$\left(\begin{bmatrix}1&0&0\\-\varepsilon^{-1}&1&0 \\\varepsilon^{-2}&-\varepsilon^{-1}&\varepsilon\end{bmatrix},
\begin{bmatrix} 1&0\\ \varepsilon^{-1}&1 \end{bmatrix}\right)$\\ 
\hline
5 & $B_{0,0}$ & $B_{\infty,0}$ & 
$\left(\diag(\varepsilon^{-2}, \varepsilon^{-1}, 1),
\begin{bmatrix} 1&0\\ \varepsilon^{-1}&1 \end{bmatrix}\right)$\\ 
\hline

6 & $B_{\infty, 1}$ & $B_{\infty,0}$ & 
$\left(\diag(1, \varepsilon, \varepsilon^2),\begin{bmatrix} 1&0\\ 0&\varepsilon^{-1} \end{bmatrix}\right)$  \\
\hline
7 & $B_{\infty, 1}$ & $C$ & 
$\left(\begin{bmatrix}1&0&0\\0&0&1 \\ 0&\varepsilon^{-1}&0\end{bmatrix},
\begin{bmatrix} 1&0\\ 0&\varepsilon \end{bmatrix}\right)$\\ 
\hline
8 & $B_{\infty,1}$ & $D$ & 
$\left(\diag(1,\varepsilon,1) , \begin{bmatrix} 1&0\\ 0&\varepsilon \end{bmatrix}\right)$\\ 
\hline
9 & $E_0$ & $E_\infty$ & 
$\left(\diag(1, \varepsilon, 1), \begin{bmatrix} 1&0\\ \varepsilon^{-1}&1 \end{bmatrix}\right)$ \\
\hline
\end{tabular}
}
\end{table*}


\newpage


\begin{thebibliography}{ccc}

\bibitem{artin} M. Artin, \textit{On Azumaya algebras and finite dimensional
representations of rings}, J. Algebra 11 (1969), 532-563.  


\bibitem{Auslander}
M. Auslander, \textit{Representation theory of finite-dimensional algebras}, Contemp. Math. 13 (1982), 27-39.

\bibitem{Belitskii} 
G. Belitskii, \textit{Normal forms in matrix spaces}, 
Integral Equations and Operator Theory 38 (2000), 251-283. 

\bibitem{Bongartz1}
K. Bongartz, \textit{Degenerations for representations of tame quivers}, Ann. scient. Ec. Norm. Sup., 4. serie, t. 28, (1995), 647-668.

\bibitem{Bongartz:Comm.Alg}
K. Bongartz, \textit{A remark on Friedland's stratification of varieties of modules}, 
Comm. Algebra 23 (1995), 2163-2165. 


\bibitem{Bongartz2} 
K. Bongartz, \textit{On degenerations and extensions of finite dimensional modules}, 
Adv. Math. 121 (1996), 245-287. 

\bibitem{Bongartz-Friedland} 
K. Bongartz, S. Friedland, 
\textit{Complete invariants for simultaneous similarity}, 	arXiv:2601.00379

\bibitem{Calderon-Lopatin} 
J. J. Calderón, A. Lopatin, 
\textit{Pairs of matrices with simple spectrum}, 
arXiv:2310.00476.

\bibitem{CH85}
R. E. Curto, D. A. Herrero, \textit{On closures of joint similarity orbits}, Integral Equations Operator
Theory 8 (1985) 489–556.

\bibitem{DKMV}
H. Derksen, I. Klep, V. Makam, J. Volčič, \textit{Ranks of linear matrix pencils separate simultaneous similarity orbits}, Advances in Mathematics. 415. 108888. 10.1016/j.aim.2023.108888, 2023.

\bibitem{Drozd}
Ju. A. Drozd, \textit{Tame and wild matrix problems, Representation theory II}, 242–258, Lecture
Notes in Math. 832, Springer, 1980.

\bibitem{Forbregd}
T. A. Forbregd, N. M. Nornes, S. O. Smalø,
\textit{Partial orders on representations of algebras},
Journal of Algebra 323 (2010) 2058–2062.

\bibitem{Friedland}
S. Friedland, \textit{Simultaneous similarity of matrices}, Adv. Math. 50 (1983) 189–265.

\bibitem{HL03}
D. Hadwin, D. R. Larson, \textit{Completely rank-nonincreasing} linear maps, J. Funct. Anal. 199 (2003) 210–227.

\bibitem{Hesselink:2} W. Hesselink, \textit{Desingularization of varieties of nullforms}, Invent. Math. 55 (1977), 141-163.


\bibitem{Hesselink:1} W. H. Hesselink, \textit{Uniform instability in reductive groups}, J. Reine Angew. Math. 304 (1978), 74-96.


\bibitem{kempf} G. Kempf, \textit{Instability in invariant theory}, Ann. Math. 108 (1978), 299-316.

\bibitem{kempf-ness} G. Kempf, L. Ness, \textit{The length of vectors in representation spaces}, Algebraic Geometry,
Lecture Notes in Mathematics 732, Springer, Berlin/Heidelberg, 1979, 233-243.

\bibitem{king} A. D. King, \textit{Moduli of representations of finite dimensional algebras}, Quart. J. Math. Oxford 45 (1994), 515-530.

\bibitem{kirwan} F. Kirwan, \textit{Cohomology of Quotients in Symplectic and Algebraic Geometry}, Mathematical Notes
31, Princeton University Press, Princeton, NJ, 1984. 


\bibitem{kraft} H. Kraft, 
\textit{Geometric methods in representation theory}, in: 
Representations of algebras, 3rd int. Conf., Puebla/Mex. 1980, Lect. Notes Math. 944 (1982), 
180-258. 

\bibitem{landsberg} 
J. M. Landsberg, \textit{Geometric complexity theory: an introduction
for geometers}, 
Ann. Univ. Ferrara 61 (2015), 65-117. 

\bibitem{lebruyn2}
L. Le Bruyn, \textit{Orbits of matrix tuples}, Alg\`ebre non commutative, groupes quantiques et invariants
(Reims, 1995), 245-261, S\'emin. Congr. 2, Soc. Math. France, Paris, 1997.

\bibitem{lebruyn} 
L. Le Bruyn, \textit{Nilpotent representations}, J. Algebra 
197 (1997), 153-177.

\bibitem{lewin} 
J. Lewin, \textit{Free modules over free algebras and free group algebras: the Schreier technique}, 
Trans. Am. Math. Soc. 145 (1969), 455-465. 

\bibitem{MFK94}
D. Mumford, J. Fogarty, F. Kirwan, \textit{Geometric invariant theory}, Ergebnisse der Mathematik und ihrer Grenzgebiete 34, Springer, 1994.

\bibitem{ness} L. Ness, \textit{A stratification of the null cone via the moment map (with an appendix by D. Mumford)},
Amer. J. Math. 106 (1984), 1281-1329.

\bibitem{Nornes}
N. M. Nornes, \textit{Degeneration and related partial
orders in representation theory}, Doctoral theses at NTNU, 2015:265.

\bibitem{Riedtmann1}
C. Riedtmann, \textit{Degenerations for representations of quivers with relations}, Ann. Sci. ´Ecole Norm. Sup. (4) 19 (1986) 275–301.

\bibitem{Smalo}
S. O. Smalø, \textit{Degenerations of representations of associative algebras}, Milan Journal of Math. 76, 1 (2008), 135-164.

\bibitem{Zwara:PAMS} G. Zwara, \textit{Degenerations for modules over representation-finite algebras}, 
Proc. Am. MAth. Soc. 127 (1999), 1313-1322. 

\bibitem{Zwara1}
G. Zwara, \textit{Degenerations of finite-dimensional modules are given by extensions},
Composito Math. 121 (2000), 205–218.

\end{thebibliography}
 \end{document}